\documentclass[11pt,a4paper]{amsart}
	
\usepackage{fullpage}
\usepackage[T1]{fontenc}

\usepackage{amsfonts}
\usepackage{amssymb}

\usepackage{euscript}

\usepackage{graphicx}
\usepackage{subeqnarray}

\newcommand{\CC}{\mathbb{C}}
\newcommand{\FF}{\mathbb{F}}

\newcommand{\NN}{\mathbb{N}}

\newcommand{\QQ}{\mathbb{Q}}
\newcommand{\RR}{\mathbb{R}}
\newcommand{\ZZ}{\mathbb{Z}}

\newcommand{\calC}{\mathcal{C}}
\newcommand{\calX}{\mathcal{X}}

\newtheorem{theorem}{Theorem}[section] 
\newtheorem{proposition}[theorem]{Proposition}
\newtheorem{lemma}[theorem]{Lemma}
\newtheorem{corollary}[theorem]{Corollary}

\newtheorem*{lemma*}{Lemma}

\theoremstyle{definition}

\theoremstyle{remark}

\newtheorem*{remark*}{Remark}
\newtheorem*{remarks*}{Remarks}

\newtheorem*{example*}{Example}

\newcommand{\card}[1]{\lvert#1\rvert}
\newcommand{\set}[2]{\left\{#1\,\mid\,#2\right\}}

\DeclareMathOperator{\Gal}{Gal}
\DeclareMathOperator{\res}{res}
\DeclareMathOperator{\Tr}{Tr}

\renewcommand{\theenumi}{\roman{enumi}}
     
\begin{document}

\title[Abelian  varieties]{On the number of points on abelian and Jacobian varieties over finite fields}

\author{Yves Aubry}
\address{Institut de Math\'ematiques de Toulon, Universit\'e du Sud Toulon-Var and
Institut de Math\'ematiques de Luminy, Aix-Marseille Universit\'e / CNRS, France}
\email{aubry@iml.univ-mrs.fr}

\author{Safia Haloui}
\address{Institut de Math\'ematiques de Luminy, Aix-Marseille Universit\'e / CNRS, France and
Department of Mathematics, Technical University of Denmark, Lyngby, Denmark}
\email{s.haloui@mat.dtu.dk}

\author{Gilles Lachaud}
\address{Institut de Math\'ematiques de Luminy, Aix-Marseille Universit\'e / CNRS , France}
\email{lachaud@univ-amu.fr}



\begin{abstract}
We give upper and lower bounds for the number of points on abelian varieties over finite fields, and lower bounds specific to Jacobian varieties. We also determine exact formulas for the maximum and minimum number of points on Jacobian surfaces.
\end{abstract}
\date{\today}
\maketitle
\tableofcontents

\section{Introduction}
\label{sec_intro}

This article has roughly a threefold aim. The first is to provide a series of upper and lower bounds for the number of points on an abelian variety defined over a finite field. A simple typical result is
$$(q + 1 - m)^{g} \leq \card{A(\FF_q)} \leq (q + 1 + m)^{g}$$
(Corollary \ref{SerreWeil} and \ref{AubryBound}). Here $A$ is an abelian variety of dimension $g$ defined over the field $\FF_{q}$ of $q$ elements, and $m$ is the integer part of $2 q^{1/2}$. This inequality improves on
$$
(q + 1 - 2 q^{1/2})^{g} \leq \card{A(\FF_q)} \leq (q + 1 - 2 q^{1/2})^{g},
$$
which is an immediate consequence of Weil's inequality. We provide as well bounds for $\card{A(\FF_{q})}$ depending on the trace of $A$. If, by chance, $A$ is the Jacobian of a curve or the Prym variety of a covering of curves, the trace is easily expressed in terms of the number of rational points on the corresponding curves. We obtain two other lower bounds depending on the \emph{harmonic mean} $\eta = \eta(A)$ of the numbers $q + 1 + x_{i}$, namely
$$
\frac{1}{\eta} = \frac{1}{g} \sum_{i = 1}^{g} \frac{1}{q + 1 + x_{i}} \, ,
$$
where
$$f_{A}(t) = \prod_{i = 1}^{g} (t^{2} + x_{i} t + q)$$
is the \emph{Weil polynomial of $A$}. Our second aim is to obtain specific lower bounds in the special case where $A = J_C$ is the Jacobian of a smooth, projective, absolutely irreducible algebraic curve $C$ defined over $\FF_q$. 
M. Martin-Deschamps and the third author proved in \cite{L-MD} that
$$\card{J_{C}(\FF_q)} \geq \eta \, \frac{q^{g-1}-1}{g}\frac{N+q-1}{q-1},$$
where $g$ is the genus of $C$ and $N = \card{C(\FF_q)}$. This article offers several improvements to this bound : for instance,
$$
\card{J(\FF_q)} \geq \Bigl( 1 - \frac{2}{q} \Bigr) \Bigl(q + 1 + \frac{N - (q + 1)}{g} \Bigr)^{g},
$$
and also
$$
\card{J_{C}(\FF_q)} \geq \frac{\eta}{g}
\left[ \binom{N + g - 2}{g - 2} +
\sum_{n = 0}^{g - 1}q^{g - 1 - n} \binom{N + n - 1}{n} \right].
$$
The third aim is to give exact values for the maximum and the minimum number of rational points on Jacobian varieties of dimension $2$, namely, to calculate, in the case $g = 2$, the numbers
$$J_q(g) = \max_{C} \card{J_{C}(\FF_q)} \quad \text{and} \quad j_q(g) = \min_{C} \card{J_{C}(\FF_q)},$$
where $C$ ranges over  the set of equivalence classes of smooth curves of genus $g$ over $\FF_q$. These numbers are the analogues for Jacobians of the numbers
$$N_q(g) = \max_{C} \card{C(\FF_q)} \quad \text{and} \quad n_q(g) = \min_{C} \card{C(\FF_q)},$$
introduced by J.-P. Serre. One has
$$(q + 1 - m)^{2} \leq j_{q}(2), \quadÊJ_{q}(2) \leq (q + 1 + m)^{2},$$
and these bounds are attained in most cases, with exceptions occurring when $q$ is special. It is worthwhile to point out that S. Ballet and R. Rolland obtained recently \cite{BalletRolland} asymptotic lower bounds on the number of points on Jacobian varieties; these results are distinct from those presented here.

The contents of this article are as follows. Section \ref{sec_abvar} is devoted to the number of points on general abelian varieties. In $\S$ \ref{subsec_upper}, we first prove an upper bound (Theorem \ref{bornetr}) obtained by H. G. Quebbemann in the case of Jacobians, and M. Perret in the case of Prym varieties. Then we state three sharper upper bounds, depending on the defect of $A$ or on a specific parameter $r$.

The lower bounds we discuss in $\S$ \ref{subsec_lower} are based on auxiliary results, predominantly on inequalities between classical means, and depend on the trace of $A$. The first one (Theorem \ref{BorneSpecht}) is based on Specht's inequality, and is symmetric to the upper bound of Theorem \ref{bornetr}. Another important result is Theorem \ref{SerreWeilTrace}, and Corollary \ref{SerreWeil} gives the unconditional lower bound stated at the beginning of this introduction.

The bounds stated in Theorem \ref{EtaBoundPure} and Proposition \ref{EtaBoundMixed} are expressed in terms of the harmonic mean $\eta$. In order to compare these bounds, we show in $\S$ \ref{subsec_lowerharmonic} that if $q \geq 8$, then $\eta \geq q + 1 - m$ (Proposition \ref{Harmonic}). The convexity method of Perret is used in $\S$ \ref{subsec_convexity} to give two other bounds in Theorem \ref{PerretOriginal} and Proposition \ref{perretAm}.

We discuss Jacobians in Section \ref{sec_jac}. The previous inequalities apply to Jacobians, depending on the number of points on the related curve, and this is stated in $\S$ \ref{subsec_JacAV}. The bounds for the number of points on Jacobians given in \cite{L-MD} depends on the identity 
$$
\frac{g}{\eta} \, \card{J_{C}(\FF_q)} =
\sum_{n = 0}^{g - 1} A_{n} + \sum_{n = 0}^{g - 2} q^{g - 1 - n} A_{n},
$$
where $A_{n}$ is the number of positive divisors on $C$ which are rational over $\FF_{q}$. In $\S$ \ref{subsec_virtual}, this identity is proved in an abstract framework, free of any geometric setting (Theorem \ref{GillesMireille3}).

By using the combinatorics of the exponential formula, various inequalities for the numbers $A_{n}$ are obtained in $\S$ \ref{subsec_specific}, depending on two conditions \eqref{ConditionB} and \eqref{ConditionN} which are satisfied by Jacobians. We discuss these conditions by giving in Propositions \ref{IneqPrimeDiv} and \ref{LargeGenusB} some results on the number $B_{n}$, which, in the case of Jacobians, is the number of rational prime cycles of degree $n$ on the curve. For instance, for a curve of genus $g$,
$$
n B_{n} \geq (q^{n/4} + 1)^{2}((q^{n/4} - 1)^{2} - 2 g).
$$
This leads to three new lower bounds. All of these are gathered in $\S$\ref{subsec_discussion}, where we compare them, and discuss their accuracy.
 
In section \ref{sec_jacsurf}, the last one, the complete calculation of $J_{q}(2)$ in Theorem \ref{Jq2} and of $j_{q}(2)$ in Theorem \ref{jq2} are worked out.

\section{Abelian varieties}
\label{sec_abvar}

\subsection{Upper bounds}
\label{subsec_upper}

Let $A$ be an abelian variety of dimension $g$ defined over the finite field  $\FF_q$ of characteristic $p$, with $q = p^{n}$. The \emph{Weil polynomial} $f_A(t)$ of $A$ is the characteristic polynomial of its Frobenius endomorphism $F_{A}$. Let $\omega_{1}, \dots, \omega_{g}, \overline{\omega}_{1}, \dots, \overline{\omega}_{g}$ be the complex roots of $f_A(t)$, with $\card{\omega_i} = q^{1/2}$ by Weil's inequality. For $1 \leq i \leq g$, we put $x_i = -(\omega_{i} + \overline{\omega}_{i})$, and we say that $A$ is \emph{of type} $[x_1,\dots ,x_g]$. The type of $A$ only depends on the isogeny class of $A$. Let
$$\tau = \tau(A) = - \sum_{i = 1}^{g} (\omega_{i} + \overline{\omega}_{i}) = \sum_{i=1}^{g} x_{i}.$$
The integer $\tau$ is the opposite of the trace of $F_{A}$, and we say that $A$ \emph{has trace} $-\tau$. The number of rational points on $A$ is $\card{A(\FF_q)} = f_A(1)$ and 
\begin{equation}
\label{card}
\card{A(\FF_q)}  = \prod_{i=1}^{g}(q+1+x_i)
\end{equation}
since 
$$f_A(t)=\prod_{i=1}^g(t-\omega_i)(t-\overline{\omega}_i)=\prod_{i=1}^g(t^2+x_it+q).$$
Since $\card{x_i} \leq 2 q^{1/2}$, one deduces from \eqref{card} the classical bounds:
$$(q+1-2 q^{1/2})^g\leq \card{A(\FF_q)} \leq(q+1+2 q^{1/2})^g.$$
The arithmetic-geometric inequality states that
$$(c_1\ldots c_k)^{1/k} \leq\frac{1}{k}(c_1+\cdots +c_k)$$
if $c_1 \dots c_k$ are non negative real numbers, with equality if and only if $c_1 = \dots = c_k$. Applying this inequality to  \eqref{card}, we obtain the following upper bound, proved by H. G. Quebbemann \cite{queb} in the case of Jacobians, and M. Perret \cite{per} in the case of Prym varieties:

\begin{theorem}
\label{bornetr}
Let $A/\FF_q$ be an abelian variety of dimension $g$ and trace $-\tau$. Then
$$\card{A(\FF_q)} \leq \Bigl(q+1+\frac{\tau}{g}\Bigr)^g,$$
with equality if and only if $A$ is of type $[x ,\dots ,x]$. \qed
\end{theorem}

Let $m = [2 q^{1/2}]$ where $[\alpha]$ denotes the integer part of the real number $\alpha$. Using the arithmetic-geometric inequality, J.-P. Serre \cite{serre0} proved that
\begin{equation}
\vert\tau\vert\leq gm, \label{trace}
\end{equation}
hence, by Theorem \ref{bornetr}:

\begin{corollary}
\label{AubryBound}
Let $A/\FF_q$ be an abelian variety of dimension $g$.Then
$$
\card{A(\FF_q)} \leq(q+1+m)^g
$$
with equality if and only if $A$ is of type $[m, \dots, m]$. \qed
\end{corollary}

We say that $A$ (or $\tau$)  has \emph{defect} $d$ if $\tau = g m - d$.

\begin{proposition}
\label{defect}
If $A$ has defect $d$, with $d = 1$ or $d = 2$, then
$$
\card{A(\FF_q)}  \leq (q + m)^{d} (q + 1 + m)^{g - d}.
$$
\end{proposition}

\begin{proof}
J.-P. Serre gives in \cite{serre2} the list of types $[x_{1}, \dots, x_{g}]$ such that $d = 1$ or $d = 2$, and we prove the proposition by inspection. The various possibilities are described in Table \ref{DefectTable} below. In this table,
$$\varphi_{1} = (- 1 + \sqrt{5})/2, \quadÊ\varphi_{2} = (- 1 - \sqrt{5})/2,$$
$$
\omega_{i} = 1 - 4 \cos^{2} \frac{i \pi}{7}, \quad i = 1, 2, 3.
$$
Moreover, $\beta_{d}$ is the right hand side of the inequality and $b = q + 1 + m$. 
\begin{table}[hbtp]
$$
\begin{array}{|c||c|c|}
\hline
d & [x_{1}, \dots, x_{g}] & \beta_{d} - \card{A(\FF_q)} \\
\hline \hline
1 & (m, \dots, m, m - 1) & 0 \\
  & (m, \dots, m, m + \varphi_{1}, m + \varphi_{2}) & b^{g - 2}\\
\hline \hline
2 & (m, \dots, m, m - 1, m - 1) & 0 \\
  & (m, \dots, m, m - 2) & b^{g - 2}\\
  & (m, \dots, m, m + \sqrt{2} - 1, m - \sqrt{2} - 1) & 2 b^{g - 2} \\
  & (m, \dots, m, m + \sqrt{3} - 1, m - \sqrt{3} - 1) & 3 b^{g - 2} \\
  & (m, \dots, m, m - 1, m + \varphi_{1}, m + \varphi_{2}) & b^{g - 3}(b - 1)\\
  & (m, \dots, m, m + \varphi_{1}, m + \varphi_{2}, m + \varphi_{1}, m + \varphi_{2}) &
    b^{g - 4}  (2 b^2 - 2 b - 1)\\
  & (m, \dots, m, m + \omega_{1}, m + \omega_{2}, m + \omega_{3}) & b^{g - 3} (2 b - 1) \\
\hline
\end{array}
$$
\caption{Types with defect $1$ or $2$, with $b = q + 1 + m$.}
\label{DefectTable}
\end{table}
\end{proof}
Now assume $g \geq 2$. The following result generalizes somehow Proposition \ref{defect}. Let
$$
y_{i} = x_{i}  - \left[ \frac{\tau}{g}\right] \ (1 \leq i \leq g), \quad
r = \sum_{i = 1}^{g} y_{i} = \tau - g\left[ \frac{\tau}{g}\right],
$$
in such a way that $r$ is the remainder of the division of $\tau$ by $g$.

\begin{proposition}
\label{bornetrAm}
If $r = 1$ or $r = g - 1$, then
$$
\card{A(\FF_q)}  \leq
{\left(q+1+\left[ \frac{\tau}{g}\right]\right)}^{g - r}{\left(q+2+\left[ \frac{\tau}{g}\right]\right)}^{r}.
$$
\end{proposition}

\begin{proof}
Take an integer $k$ with $1 \leq k \leq g - 1$. If $H$ belongs to the set $\mathfrak{P}_{k}$ of subsets of $\{1, \dots, g\}$ with $k$ elements, we define
$$
y_{H} = \sum_{i \in H} y_{i} \quad \text{and} \ f_{k}(T) = \prod_{H \in \mathfrak{P}_{k}} (T - y_{H}).
$$
The polynomials $f_{k}$ are in $\ZZ[T]$, since the family $(x_{i})$ is stable under $\Gal(\bar{\QQ} / \QQ)$. Moreover,
$$
\frac{\Tr y_{H}}{\deg f_{k}} = \frac{1}{\binom{g}{k}} \sum_{H \in \mathfrak{P}_{k}} y_{H} = 
\frac{1}{\binom{g}{k}}\binom{g-1}{k-1}\sum_{i=1}^{g} y_{i} = \frac{k r}{g}.
$$
Now recall that, if $y$ is a totally positive algebraic integer, then the arithmetic-geometric inequality implies that
$$
\Tr y \geq \deg y.
$$
Hence, if $y_{H} > 0$ for every $H \in \mathfrak{P}_{k}$, then $k r \geq g$. This shows that if $k r < g$, then, after renumbering  the numbers $x_{i}$ if necessary, we have
$$
\sum_{i=1}^{k} y_{i} \leq 0, \quad \text{i.e.} \
\sum_{i=1}^{k} x_{i} \leq k \left[ \frac{\tau}{g}\right].
$$
Now choose $k = g - r$. Then
$$
\sum_{i = g - r + 1}^{g} x_{i} \geq r(\left[ \frac{\tau}{g}\right] +1).
$$
Hence, according to the arithmetic-geometric inequality,
\begin{eqnarray*}
\card{A(\FF_q)}  = \prod_{i=1}^g(q+1+x_i)
& \leq &
\left(q + 1 + \frac{1}{g - r} \sum_{i = 1}^{g - r}x_i\right)^{g - r}
\left(q + 1 + \frac{1}{r} \sum_{i = g - r + 1}^{g} x_i \right)^{r}\\
& \leq &
\left(q+1+\left[ \frac{\tau}{g}\right]\right)^{g - r}
\left(q+2+\left[ \frac{\tau}{g}\right]\right)^{r},
\end{eqnarray*}
where the second inequality follows from Lemma \ref{Convexity} below.
To complete the proof of  Proposition \ref{bornetrAm}, it remains to establish that 
$r (g - r) < g$ if and only if $r = 1$ or $r = g - 1$. Observe that the inequality $r (g - r) < g$ holds in every case if $g \leq 3$. Now assume that $g \geq 4$ and let
$$
r_{\pm}(g) = \frac{1}{2} (g \pm (g^{2} - 4 g)^{1/2}).
$$
The inequality holds if and only if $r < r_{-}(g)$ or $r > r_{+}(g)$. If $g = 4$, then $r_{-}(4) = r_{+}(4) = 2$. If $g \geq 5$, then $1 < r_{-}(g) < 2$ and $g - 2 < r_{+}(g) < g - 1$.
\end{proof}

\begin{lemma}
\label{Convexity}
Let $0 \leq a \leq c \leq d \leq b$. If $(g - r)a + rb=(g - r)c + rd$, then
$$a^{g-r} b^r\leq c^{g-r} d^r.$$
\end{lemma}

\begin{proof}
The barycenter of $(a, \log a)$ and  $(b, \log b)$ with the weights  $g-r$ and $r$ is
$$\Bigl(\frac{(g - r)a + r b}{g},\frac{(g - r)\log a + r \log b}{g}\Bigr)$$
and that of $(c, \log c)$ and $(d, \log d)$ with the same weights is   
$$\Bigl(\frac{(g-r)c+rd}{g},\frac{(g-r)\log c+r\log d}{g}\Bigr) = \Bigl(\frac{(g - r)a + r b}{g},\frac{(g - r)\log c + r \log d}{g}\Bigr),$$
and the result follows from the concavity of the logarithm.
\end{proof}
If $\tau$ has defect $1$, then
$$
\frac{\tau}{g} = m - \frac{1}{g}, \quad \left[ \frac{\tau}{g}\right] = m - 1, \quad r =  (g m - 1) - (g m - g) = g - 1,
$$
and Proposition \ref{bornetrAm} reduces to Proposition \ref{defect}.

\begin{corollary}
If $\tau = g m - g + 1$ (defect $g - 1$), then
$$
\card{A(\FF_q)}  \leq (q + m)^{g - 1} (q + 1 + m).
$$
\end{corollary}

\begin{proof}
Here
$$
\frac{\tau}{g} = m -  1 + \frac{1}{g}, \quad \left[ \frac{\tau}{g}\right] = m - 1, \quad r =  g m - g + 1 - (g m - g) = 1
$$
and the result follows.
\end{proof}

\begin{remark*}
Smyth's Theorem \cite[p. 2]{Smyth} asserts that if $x$ is a totally positive algebraic integer, then with finitely many exceptions, explicitly listed,
$$
\Tr x \geq 1.7719 \deg x.
$$
From this one deduces that the conclusion of Proposition \ref{bornetrAm} holds true for every $r$ if $g \leq 7$ and if the polynomials $x - 1$ and $x^{2} - 3 x + 1$ does not divide $f_{g - r}$.
\end{remark*}

\subsection{Lower bounds}
\label{subsec_lower}

The first lower bound for $\card{A(\FF_q)}$ is symmetrical to the upper bound given in Theorem \ref{bornetr}, and depends on \emph{Specht's ratio}, defined for $h \geq 1$ as
$$
S(h) = \frac{h^{1/(h - 1)}}{e \log h^{1/(h - 1)}}, \quad S(1) = 1.
$$
It is the least upper bound of the ratio of the arithmetic mean to the geometric mean of numbers $c_{1}, \dots, c_{n} \in [a,b]$ with $0 < a < b$ and $h = b/a > 1$, see \cite{Specht}, \cite{Fujii}. That is, the following \emph{Specht's inequality} holds:
$$
\frac{c_{1} + \dots + c_{n}}{n} \leq S(h) (c_{1} \dots c_{n})^{1/n},
$$
as a reverse of the arithmetic-geometric inequality. 

\begin{theorem}
\label{BorneSpecht}
If $q \geq 2$, let $M(q) = 1/S(h(q))$, where $S(h)$ is Specht's ratio and
$$
h(q) = \left(\frac{q^{1/2} + 1}{q^{1/2} - 1}\right)^{2}.
$$
Let $A/\FF_q$ be an abelian variety of dimension $g$ and trace $-\tau$ over $\FF_{q}$.
\begin{enumerate}
\item
\label{SPB1}
We have
$$
\card{A(\FF_q)} \geq M(q)^{g} \Bigl(q + 1 + \frac{\tau}{g} \Bigr)^{g}.
$$
\item
\label{SPB2}
In particular,
$$
\card{A(\FF_q)} \geq (1 - \frac{2}{q})^{g} \Bigl(q + 1 + \frac{\tau}{g} \Bigr)^{g},
$$
where $1 - 2/q$ has to be replaced by $0.261$ If $q = 2$. \qed
\end{enumerate}
\end{theorem}

\begin{proof}
According to \eqref{card}, we apply Specht's inequality with $c_{i} = q + 1 + x_{i}$, $1 \leq i \leq g$. Then
$$
h = h(q), \quad
c_{1} \dots c_{g}  = \card{A(\FF_{q})}, \quad
\frac{c_{1} + \dots + c_{g}}{g} = q + 1 + \frac{\tau}{g},
$$
and this gives \eqref{SPB1}. The function $M(q)$ is increasing, and $M(2) = 0.261 \dots$. Also,
$$
M(q) = 1 - \frac{2}{q} + \frac{10}{9 q^{2}} + O(q^{-3})  \quad \text{if} \quad q \rightarrow \infty,
$$
and one checks that
$$
M(q) \geq 1 - \frac{2}{q} \quad \text{if} \quad q \geq 2.
$$
Hence, \eqref{SPB2} follows from \eqref{SPB1}.
\end{proof}

Theorems \ref{bornetr} and \ref{BorneSpecht}\eqref{SPB2} are summarized in the relation
\begin{equation}
\label{TraceBounds}
(1 - \frac{2}{q}) (q + 1 + \frac{\tau}{g}) \leq \card{A(\FF_q)}^{1/g} \leq q + 1 + \frac{\tau}{g} \, .
\end{equation}

It is natural to ask whether $\card{A(\FF_q)}$ has a lower bound symmetrical to Corollary \ref{AubryBound}, and the answer turns out to be in the affirmative.

\begin{lemma}
\label{Basic}
Let $\lambda_{1}, \dots, \lambda_{n}$ be non-negative real numbers, and
$$
F(T) = \prod_{i = 1}^{n} (T + \lambda_{i}) \in  \RR[T].
$$
Let
$$
\pi = \left( \prod_{i = 1}^{n} \lambda_{i} \right) ^{1/n}, \quad
\sigma = \frac{1}{n} \sum_{i = 1}^{n} \lambda_{i},
$$
and assume that $t \geq 0$ and $0 \leq \lambda \leq \pi$.
\begin{enumerate}
\item
\label{Holder1}
We have
$$
F(t) \geq (t + \lambda)^{n} + n (\sigma - \lambda) t^{n - 1},
$$
and in particular $F(t) \geq (t + \lambda)^{n}$.
\item
Moreover
\label{Holder2}
$$
F'(t) \geq n (t + \lambda)^{n - 1} + n(n - 1) (\sigma - \lambda) t^{n - 2},
$$
and in particular $F'(t) \geq n (t + \lambda)^{n - 1}.$
\end{enumerate}
These inequalities are strict unless $\lambda_{1} = \dots = \lambda_{n}$ and $\lambda = \pi$.
\end{lemma}

\begin{proof}
We exclude the case where $\lambda_{1} = \dots = \lambda_{n}$. Let $\mathfrak{P}_{k}$ be the set of subsets of $\{1, \dots, n\}$ with $k$ elements, and put
$$p_{k} = \prod_{H \in \mathfrak{P}_{k}} \prod_{i \in H} \lambda_{i}.$$
For $1 \leq i \leq n$, the coefficient of $\lambda_{i}$ in $p_{k}$ is equal to the number of subsets in $\mathfrak{P}_{k}$ containing $i$, and this number is equal to $\binom{n - 1}{k - 1}$. Therefore
$$
p_{k} = \prod_{i = 1}^{n} \lambda_{i}^{\binom{n - 1}{k - 1}} = \pi^{\binom{n - 1}{k - 1} n} \geq
\lambda^{\binom{n - 1}{k - 1} n}, \quad \text{hence} \quad p_{k}^{1/\binom{n}{k}} \geq \lambda^{k}.
$$
For $0 \leq k \leq n$, let
$$s_{k} = \sum_{H \in \mathfrak{P}_{k}} \prod_{i \in H} \lambda_{i}$$
be the \emph{elementary symmetric function} of degree $k$ in the numbers $\lambda_{1}, \dots, \lambda_{n}$. The arithmetic-geometric inequality implies
$$
p_{k}^{1/\binom{n}{k}} < \binom{n}{k}^{- 1} s_{k}, \quad \text{hence} \quad
\binom{n}{k} \lambda^{k} < s_{k}.
$$
Putting $s_{0} = 1$, and since $s_{1} = n \sigma$, we obtain
\begin{eqnarray*}
F(t) = \sum_{k = 0}^{n} s_{k} t^{n - k} & = & \sum_{k = 0}^{n} \binom{n}{k} \lambda^{k} t^{n - k} +
\sum_{k = 0}^{n} \left( s_{k} - \binom{n}{k} \lambda^{k} \right) t^{n - k}
\\ & > &
(t + \lambda)^{n} + n (\sigma - \lambda) t^{n - 1},
\end{eqnarray*}
which proves the first inequality of \eqref{Holder1}, and of course the second, since $\sigma - \lambda \geq 0$. Also,
$$F'(t) = \sum_{k = 0}^{n - 1} (n - k) s_{k} t^{n - 1 - k},$$
and \eqref{Holder2} is proved along the lines of the proof of \eqref{Holder1}.
\end{proof}

\begin{remark*}
With the notation of Lemma \ref{Basic}, the basic inequality
$$
F(t) \geq (t + \lambda)^{n}
$$
is just a consequence of the following H\"older's inequality, where $x_{1}, \dots, x_{n}$ and  $y_{1}, \dots, y_{n}$ are non negative real numbers:
$$
\prod_{i = 1}^{n} (x_{i} + y_{i}) >
\left( \prod_{i = 1}^{n} x_{i}^{1/n} + \prod_{i = 1}^{n} y_{i}^{1/n} \right)^{n},
$$
unless $x_{1} = \dots = x_{n}$ and $y_{1} = \dots = y_{n}$.
\end{remark*} 

The \emph{real Weil polynomial} of $A$ is
$$h_{A}(t) = \prod_{i = 1}^{g}(t + x_{i}).$$
Then $f_{A}(t) = t^{g} h_{A}(t + q t^{-1})$ and $\card{A(\FF_q)} = h_{A}(q + 1)$.

\begin{theorem}
\label{SerreWeilTrace}
Let $A/\FF_q$ be an abelian variety of dimension $g$ and trace $-\tau$. Then
$$ \card{A(\FF_q)}  \geq (q + 1 - m)^{g} + (q - m)^{g - 1}  (g m + \tau).$$
\end{theorem}

\begin{proof}
We apply Lemma \ref{Basic} to the polynomial
$$F(t) = h_{A}(t + m + 1) = \prod_{i = 1}^{g}(t + m + 1 + x_{i}).$$
Here $\lambda_{i} = m + 1 + x_{i}$, and
$$\pi = \prod_{i = 1}^{g} (m + 1 + x_{i})^{1/g} > 0.$$
Then $\pi^{g} \in \ZZ$, since this number is left invariant by $\Gal(\bar{\QQ} / \QQ)$, hence $\pi \geq 1$. 
Now apply Lemma \ref{Basic}\eqref{Holder1} with $\lambda = 1$ and $t = q - m$. We get
$$
h_{A}(q + 1) = F(q - m) \geq (q - m + 1)^{g} + g (q - m)^{g - 1}  (\sigma - 1),
$$
and the result follows by observing that $g(\sigma - 1) = g m + \tau$.
\end{proof}

\begin{remark*}
For $g = 1$, the inequality of Theorem \ref{SerreWeilTrace} is an equality (if $q\leq 4$ then $q - m = 0$ and we use the convention $0^{0} = 1$).
\end{remark*}

Serre's inequality \eqref{trace} implies $g m + \tau \geq 0$, hence :

\begin{corollary}
\label{SerreWeil}
Let $A/\FF_{q}$ be an abelian variety of dimension $g$. Then
$$\card{A(\FF_q)} \geq (q + 1 - m)^{g},$$
with equality if and only if $A$ is of type $[- m, \dots,- m]$.
\qed
\end{corollary}

The \emph{harmonic mean} $\eta = \eta(A)$ of the numbers $q + 1 + x_{i}$ is defined by
$$
\frac{1}{\eta} = \frac{1}{g} \sum_{i = 1}^{g} \frac{1}{q + 1 + x_{i}} \, ,
$$
hence
$$
\frac{1}{\eta} = \frac{1}{g} \sum_{i= 1}^{g} \frac{1}{\vert 1-\omega_{i}\vert ^2} \, .
$$
The classical inequality between harmonic and geometric means leads to:

\begin{theorem}
\label{EtaBoundPure}
Let $A/\FF_q$ be an abelian variety of dimension $g$. Then
$$\card{A(\FF_q)} \geq \eta^{g}. \rlap \qed$$
\end{theorem}

The next proposition is similar to Theorem \ref{SerreWeilTrace}.

\begin{proposition}
\label{EtaBoundMixed}
Let $A/\FF_q$ be an abelian variety of dimension $g$. Then
$$
\card{A(\FF_q)} \geq \eta (q + 1 - m)^{g - 1} +  \eta \, \frac{g - 1}{g} (q - m)^{g - 2} (g m + \tau).
$$
\end{proposition}

\begin{proof}
Since
$$
\frac{h'_{A}(t)}{h_{A}(t)} = \sum_{i = 1}^{g} \frac{1}{t + x_{i}} \, ,
$$
we have
$$\card{A(\FF_q)} = \frac{\eta}{g}Ê\, h'_{A}(q + 1).$$
By applying Lemma \ref{Basic}\eqref{Holder2} with $F(t)$ as in the proof of Theorem \ref{SerreWeilTrace}, we obtain
$$
h'_{A}(q + 1) \geq g (q + 1 - m)^{g - 1} + (g - 1) (q - m)^{g - 2} (g m + \tau).
$$
leading to the second inequality.
\end{proof}

\subsection{Lower bound for the harmonic mean}
\label{subsec_lowerharmonic}
In order to compare the results of the previous section, a lower bound for $\eta$ is needed. The proof of the forthcoming proposition makes use of a simple observation, following C. Smyth \cite{Smyth} and J.-P. Serre \cite{SerreBertrand}: let
$$P(t) = \prod_{i = 1}^{g}(t - \alpha_{i}) \in \ZZ[t]$$
and $F \in \ZZ[t]$. The resultant $\res(P,F)$ is an integer, and if $\res(P,F) \neq 0$, then
\begin{equation}
\label{SmythLog}
\sum_{i = 1}^{g} \log \card{F(\alpha_{i})} = \log \card{\res(P,F)} \geq 0.
\end{equation}

\begin{proposition}
\label{Harmonic}
If $q \geq 8$, then
$$
\eta(A) \geq q + 1 - m .
$$
\end{proposition}

\begin{proof}
(a)
\begin{itshape}
Assume $c > 2$. If $- 1 < t \leq c(c - 2)$, then
$$
\log(1 + t) \leq \frac{c t}{c + t} \, . 
$$
\end{itshape}
Indeed, if
$$
f(t) = \frac{c t}{c + t} - \log(1 + t)$$
then $$f'(t) = \frac{t(c^{2} - 2 c - t)}{(1 + t)(c+ t)^{2}}.
$$
Thus $f'(t) \leq 0$ if $-1 \leq t < 0$ and $f'(t) \geq 0$ if $0 \leq t \leq c(c - 2)$. This proves our statement, since $f(-1) = + \infty$ and $f(0) = 0$.

\bigskip

(b)
\begin{itshape}
If $q \geq 8$, and if $\card{x} \leq m + 1$, then
$$
\log (x + m + 1) \leq ( q + 1 - m) \left(1 - \frac{q + 1 - m}{q + 1 + x} \right).
$$
\end{itshape}
Let $c = q + 1 - m$. If $q > 5$, then $q + 1 - m > 2$, and if $q \geq 11$, then
$$
2 m + 1 \leq (q + 1 - m)(q - 1 - m).
$$
Putting $t = x + m$, the hypotheses of (a) are satisfied, and
$$
\log (x + m + 1) \leq \frac{(q + 1 - m)(x + m)}{q + 1 + x}
= ( q + 1 - m) \left(1 - \frac{q + 1 - m}{q + 1 + x} \right).
$$

\bigskip

(c)
Now use \eqref{SmythLog} with $F(t) = t + m + 1$ and $P(t) = (-1)^{g} h_{A}(-t)$. Obviouly $\res(P,F) \neq 0$, hence
$$
\sum_{i = 1}^{g} \log (x_{i} + m + 1) \geq 0.
$$
One deduces from (b) that
$$
\sum_{i = 1}^{g} \left(1 - \frac{q + 1 - m}{q + 1 + x_{i}} \right) \geq 0, \quad  \text{hence} \quad
(q + 1 - m) \sum_{i = 1}^{g} \frac{1}{q + 1 + x_{i}} \leq g.
$$
Finally, the conclusion is obvious if $q = 9$ (a square), and verified by brute force if $q = 8$.
\end{proof}

\begin{remarks*}
(i)
Proposition \ref{Harmonic} shows that Theorem \ref{EtaBoundPure} is stronger that Theorem \ref{SerreWeil} if $q \geq 8$.

(ii)
It can happen that $\eta < q + 1 - m$. Take $q = 2$ for instance. As explained in section \ref{sec_jacsurf}, there is an abelian surface $A/\FF_{2}$ with
$$
f_{A}(t) = t^{4} - t^{3} - 2 t + 4, \quad \text{and} \quad \card{A(\FF_{2}} = 2.
$$
Hence, $A$ is of type $[x_{+},x_{-}]$, with $x_{\pm} = (- 1 \pm \sqrt{17})/2$, and $\eta = 4/5$, but $q + 1 - m = 1$.
\end{remarks*}

\subsection{Convexity}
\label{subsec_convexity}

Another way to obtain lower bounds for $\card{A(\FF_q)} $ is to use convexity methods as performed by M. Perret. We give the statement of \cite[Th. 3]{per}, slightly rectified.

\begin{theorem}
\label{PerretOriginal}
Let $A/\FF_q$ be an abelian variety of dimension $g$ and trace $- \tau$. Then
$$
\card{A(\FF_q)}  \geq (q-1)^g\Bigl(\frac{ q^{1/2}+1}{ q^{1/2}-1}\Bigr)^{\omega - 2\delta}, \quad
\text{where} \ \omega = \frac{\tau}{2 q^{1/2}}, 
$$
and where $\delta = 0$ if $g + \omega$ is an even integer, and $\delta = 1$ otherwise.
\end{theorem}
 
\begin{proof}
The idea is to find the minimum of the function
$$(x_1,\dots ,x_g)\mapsto \prod_{i=1}^{g}(q+1+x_i)$$
on the set
$$\set{ (x_1,\dots ,x_g)\in[-2 q^{1/2}, 2 q^{1/2} ]^g}{x_1 + \dots + x_g =\tau}.$$
Let
$$
y_i=\frac{x_i}{2 q^{1/2}}, \quad c = \frac{q+1}{2 q^{1/2}}.
$$
The problem  reduces to minimizing the function  
$$F(y_1,\dots ,y_g) = \sum_{i=1}^{g}\log (c+y_i)$$
on the polytope
$$P = \set{(y_1,\dots ,y_g)\in[-1, 1 ]^g}{y_1 + \dots + y_g = \omega}.$$
The set of points in $P$ where $F$ is minimal is invariant under permutations. Since $F$ is strictly concave, the points in this set are vertices of $P$. But at most one of the coordinates of a vertex of $P$ is different from $\pm 1$. Hence the minimum of $F$ is attained at a vertex
$$
\gamma = (1, \dots, 1, -1, \dots, -1,\beta), \quad \text{with} \ \beta \in [-1, 1].
$$
Denote by  $u$ and $v$ the respective numbers of occurrences of $1$ and of $-1$ in $\gamma$, and set $\delta =1$ if $\beta$ is in the open interval $(-1, 1)$ and $0$ otherwise. Then
$$
u+v+\delta = g, \quad
u-v+\delta\beta = \omega,
$$
and adding these equations, it follows that if $\delta = 0$ then $g + \omega$ is an even integer. The converse is true: if $\delta =1$ then $\beta$ is in the open interval $(-1, 1)$, and either $\beta\neq 0$ and $g+\omega$ is not an integer, or $\beta =0$ and $g+\omega =2u+1$. Hence,
\begin{eqnarray*}
\min_{(y_1,\dots ,y_g)\in P} \exp F(y_1,\dots ,y_g)& = & (c+\beta)^{\delta}(c+1)^u(c-1)^v\\
& = &  (c+\beta)^{\delta}(c^2-1)^{\frac{u+v}{2}}{\Bigr(\frac{c+1}{c-1}\Bigl)}^{\frac{u-v}{2}}\\
& = &  (c+\beta)^{\delta}(c^2-1)^{\frac{g-\delta}{2}}{\Bigr(\frac{c+1}{c-1}\Bigl)}^{\frac{\omega-\delta\beta}{2}},
\end{eqnarray*}
and $c+\beta\geq c-1$ and $\omega-\delta\beta\geq \omega-\delta$.
\end{proof}

It is possible to improve Theorem \ref{PerretOriginal} by computing more explicitly the coordinates of the extremal points of $P$ in the above proof.

\begin{proposition}
\label{perretAm}
Let
$$
r = \left[\frac{g+\left[\omega\right]}{2}\right],   \quad
s = \left[\frac{g-1-\left[\omega\right]}{2}\right], \quad \text{where} \ \omega = \frac{\tau}{2 q^{1/2}}.
$$
Then
$$
\card{A(\FF_q)}  \geq (q+1+\tau -2(r-s) q^{1/2})(q+1+2 q^{1/2})^r(q+1-2 q^{1/2})^s.
$$
\end{proposition}

\begin{proof}
We keep the notation and results from the proof of Theorem \ref{PerretOriginal}. If $\gamma\neq (1,\dots ,1)$, we denote by $r$ and $s$ the number of occurrences of $1$ and of $-1$ respectively in  $\gamma$ but now, without eventually counting  $\beta$. Then $r - s = \omega - \beta$, thus $\beta$ must be equal to $\{\omega\}=\omega-[\omega]$ or $\{\omega\}-1$ (after perhaps a permutation of $\beta$ with one of the coordinate equal to   $-1$ in the case where $\beta= 1$). Thus,
$$
r+s = g-1, \quad
r-s = [\omega]+\epsilon, \quad
\beta =\{\omega\}-\epsilon,
$$
where $\epsilon\in\{ 0,1\}$. If $\gamma = (1,\dots ,1)$, the previous identities remain true if $r = g$ and $s = -1$. The equations  $2r=g-1+[\omega]+\epsilon$ and $2s=g-1-[\omega]-\epsilon$ show that  $\epsilon =1$ if and only if  $g+[\omega]$ is even, and  that
$$r  =  \Bigr[\frac{g+[\omega]}{2}\Bigl]\quad\mbox{ and }\quad s  =  \Bigr[\frac{g-1-[\omega]}{2}\Bigl].$$
Proceeding as in the proof of Theorem \ref{PerretOriginal}, we obtain 
$$\min_{(y_1,\dots ,y_g)\in P} \exp F(y_1,\dots ,y_g) =   (c+\{\omega\}-\epsilon)(c^2-1)^{\frac{g-1}{2}}{\Bigr(\frac{c+1}{c-1}\Bigl)}^{\frac{[\omega]+\epsilon}{2}}.$$
Then
$$
\card{A(\FF_q)}  \geq (q-1)^{g-1} 
\Bigr(q + 1 + 2 q^{1/2} (\{\omega\}-\epsilon) \Bigl)
\Bigr(\frac{ q^{1/2}+1}{q^{1/2}-1}\Bigl)^{\frac{\left[\omega\right]+\epsilon}{2}}
$$
where $\epsilon =1$ if $g+ [\omega]$ is even and $0$ otherwise, from which the result follows.
\end{proof}

\begin{remark*}
If $q$ is not a square, the bound of  Proposition \ref{perretAm} is reached only if 
$r=s$ (which implies that  $\vert \tau\vert < 2 q^{1/2}$) and $\tau$ is the trace of some elliptic curve.   If $q$ is a square, this bound is not reached  only if $\tau -2(r-s) q^{1/2}$ is not the trace of an elliptic curve  (in particular, it is reached if  $\tau$ is coprime to $p$).
\end{remark*}

\section{Jacobians}
\label{sec_jac}

\subsection{Jacobians as abelian varieties}
\label{subsec_JacAV}

Let $C$ denotes a nonsingular, projective, absolutely irreducible curve defined over $\FF_{q}$, with Jacobian $J_{C}$. We define
$$J_q(g)=\max_{C} \card{J_C(\FF_q)} \quad \mbox{ and }\quad j_q(g)=\min_{C} \card{J_C(\FF_q)},$$
where $C$ ranges over  the set of curves of genus $g$. Theorem \ref{SerreWeil} implies
$$(q + 1 - m)^{g} \leq j_{q}(g) \leq \card{J_C(\FF_{q})} \leq J_{q}(g) \leq (q+1+m)^g.$$
Let $N = \card{C(\FF_q)}$ be  the number of $\FF_q$-rational points on $C$. If $J_C$ has trace $- \tau$, then
$N = q + 1 + \tau$, and \eqref{TraceBounds} implies
$$
\Bigl( 1 - \frac{2}{q} \Bigr) \Bigl(q + 1 + \frac{N - (q + 1)}{g} \Bigr)^{g} \leq
\card{J_C(\FF_q)} \leq \Bigl(q + 1 + \frac{N - (q + 1)}{g} \Bigr)^{g}.$$
Hence,
$$
\Bigl( 1 - \frac{2}{q} \Bigr) \Bigl( q + 1 + \frac{N_{q}(g) - (q + 1)}{g} \Bigr)^{g} \leq
J_{q}(g) \leq \Bigl( q + 1 + \frac{N_{q}(g) - (q + 1)}{g} \Bigr)^{g},
$$
where $N_q(g)$ stands for the maximum number of rational points on a curve defined over $\FF_q$ of genus $g$. The quantity  $J_q(g)$ has the following asymptotic behaviour. Define, as usual,
$$
A(q) = \limsup_{g\rightarrow\infty} \frac{N_{q}(g)}{g} \, .
$$
The preceding inequalities imply
$$
(1 - \frac{2}{q}) (q + 1 + A(q)) \leq \limsup J_{q}(g)^{1/g} \leq q + 1 + A(q).
$$
On the one hand, the  Drinfeld-Vl\u{a}du\c{t} upper bound \cite[p. 146]{tvs}
$$A(q) \leq q^{1/2} - 1$$
implies
$$\limsup_{g\rightarrow\infty}( J_q(g))^{1/g}\leq q + q^{1/2}$$
(the Weil bound would only give the upper bound $q+1+ 2 q^{1/2}$).  On the other hand, S. Vl\u{a}du\c{t} has proved \cite{vla} that if $q$ is a square, then  
$$q(\frac{q}{q-1})^{q^{1/2}-1} \leq \limsup_{g\rightarrow\infty}( J_q(g))^{1/g}.$$
Observe that, when  $q \rightarrow \infty$,
$$q(\frac{q}{q-1})^{q^{1/2}-1}=q+q^{1/2}-\frac{1}{2} + O(\frac{1}{q^{1/2}}).$$ 

\begin{remark*}
Some observations are worthwile to point out on the relationship between $N_q(g)$ and $J_q(g)$. The number of points on the Jacobian of a maximal curve, that is, with $N_q(g)$ points, does not necessarily reach $J_q(g)$. For instance, J.-P. Serre  \cite[p. Se47]{serre2} has shown that there exist two curves of genus $2$ over $\FF_3$ with $N_3(2)=8$ points whose Jacobians have $35$ points and $36$ points respectively.

We shall see in Section \ref{sec_jacsurf} that a maximal Jacobian surface, that is, with $J_q(2)$ points, is always the Jacobian of a maximal curve (but there is no reason that this could remain true if $g>2$).

A curve reaching the Serre-Weil bound (i.e. with $q+1+m$ points) has type $[m,\dots ,m]$ by \eqref{trace}, hence, in the case where the Serre-Weil bound is reached for curves of genus $g$, a curve of genus $g$ is maximal if and only if its Jacobian is maximal.
\end{remark*}

We record the consequences for Jacobians of the results of section \ref{sec_abvar}. Theorem \ref{BorneSpecht}\eqref{SPB1} implies, with $M(q)$ as defiend there:

\begin{proposition}
\label{BorneSpechtJac}
If $C$ is a curve of genus $g$ over $\FF_q$ as above, then
\begin{equation}
\label{bsj}
\tag{\textbf{I}}
\card{J_C(\FF_q)} \geq
M(q)^{g} \Bigl(q + 1 + \frac{N - (q + 1)}{g} \Bigr)^{g}.
\end{equation}
\end{proposition}

Proposition \ref{SerreWeilTrace} implies
$$\card{J_C(\FF_q)} \geq (q + 1 - m)^g + (g m + N - q - 1)(q - m)^{g - 1},$$
and Proposition \ref{perretAm} leads to:

\begin{proposition}
\label{firstbound}
If $C$ is a curve of genus $g$ over $\FF_q$ as above, then
\begin{equation}
\label{Borne-per}
\tag{\textbf{II}}
\card{J_C(\FF_q)} \geq (N - 2(r - s) q^{1/2})(q + 1 + 2  q^{1/2})^r(q + 1 - 2 q^{1/2})^{s},
\end{equation}
with
$$
r = \left[\frac{g+\left[\omega\right]}{2}\right],   \quad
s = \left[\frac{g-1-\left[\omega\right]}{2}\right], \quad \text{where} \ \omega = \frac{N - q - 1}{2 q^{1/2}}.
\rlap \qed
$$
\end{proposition}

\subsection{Virtual zeta functions}
\label{subsec_virtual}

The usual numerical sequences associated to a curve over $\FF_{q}$ (number of points, number of effective and prime divisors) occur in expressions involving the numerator of its zeta function. These sequences are defined and their main properties are obtained here by taking as starting point a polynomial satisfying suitable conditions, without any geometric reference. If $A$ is an abelian variety of dimension $g$ over $\FF_{q}$, with characteristic polynomial $f_{A}(t)$, the reciprocal Weil polynomial $P(t) = P_{A}(t) = t^{2 g} f_{A}(t^{- 1})$ fulfills the following conditions :
\begin{enumerate}
\item
\label{eqfct1}
$P(t)$ satisfies the functional equation
\begin{equation*}
\label{eqfct11}
P(t) = q^{g} t^{2 g} P(\frac{1}{q t}),
\end{equation*}
\item
\label{eqfct2}
$P(0) = 1$ and $P(1) \neq 0$.
\end{enumerate}
The inverse roots of $P$ have modulus $q^{1/2}$, but we do not need this condition now. Condition \eqref{eqfct1} means that if
$$
P(t) = \sum_{n = 0}^{2 g} a_{n} t^{n}
$$
then $a_{2 g - n} = q^{g - n}a_{n}$ for $0 \leq n \leq 2 g$, in other words the sequence $q^{- n/2} a_{n}$ is \emph{palindromic}. This condition implies that the possible multiplicity of $\pm q^{1/2}$ is even, as shown by H. Stichtenoth in the proof of \cite[Th. 5.1.15(e)]{Stichtenoth}.
Now, let $q \in \NN$ with $q \geq 2$, and denote by $\mathbf{P}_{g}(q)$ the set of polynomials of degree $2 g$ in $\ZZ[t]$ satisfying \eqref{eqfct1} and \eqref{eqfct2}. If $P \in \mathbf{P}_{g}(q)$, the \emph{virtual zeta function} with numerator $P$ is the power series
\begin{equation}
\label{zeta}
Z(t) = Z_{P}(t) = \frac{P(t)}{(1 - t)(1 - q t)} \in \ZZ[[t]] \, ,
\end{equation}
which is convergent if $\card{t} < q^{- 1}$. For $n \geq 0$, define $A_{n} = A_{n}(P) \in \ZZ$ and $N_{n} = N_{n}(P) \in \QQ$ as the coefficients of the power series
\begin{equation}
\label{expressionzeta}
Z(t) = \sum_{n = 0}^{\infty} A_{n} t^{n}, \quad
\log Z(t) = \sum_{n = 1}^{\infty} N_{n} \, \frac{t^{n}}{n},
\end{equation}
and put
$$B_{n} = B_{n}(P) = \frac{1}{n} \sum_{d | n} \mu(\frac{n}{d}) N_{d} \in \QQ,$$
where $\mu$ is the M\"{o}bius function. The M\"{o}bius inversion formula implies
\begin{equation}
\label{Moebius}
N_{n} = \sum_{d | n} d B_{d}.
\end{equation}

We recall now a result of Schur \cite{Schur}:

\begin{lemma}
\label{ZetaSch}
Let $P \in \mathbf{P}_{g}(q)$. With the previous notation:
\begin{enumerate}
\item
\label{ZSC1}
We have
$$
Z(t) = \prod_{n = 1}^{\infty}(1 - t^{n})^{-B_{n}}.
$$
\item
\label{ZSC2}
The numbers $B_{n}$ and $N_{n}$ are in $\ZZ$ for any $n \geq 1$.
\item
\label{ZSC3}
If $n \geq 1$, then
\begin{equation}
\label{AnFormula}
A_{n} = \prod_{b \in \EuScript{P}_{n}} \prod_{i = 1}^{n} \binom{B_{i} + b_{i} - 1}{b_{i}},
\end{equation}
with
$$\EuScript{P}_{n} =\{b=(b_1,\ldots,b_n)\in \NN^n \mid b_1+2b_2+\cdots+nb_n=n\}.$$
\end{enumerate}
\end{lemma}

\begin{proof}
First, \eqref{ZSC1} is a consequence of \eqref{Moebius}. By using the negative binomial formula 
$$(1 - t)^{- m} = \sum_{n = 0}^{+ \infty} \binom{m + n - 1}{n}t^{n},$$ 
where $m \in \CC$ and
$$
\binom{r}{0} = 1, \quadÊ\binom{r}{k} = \frac{r(r-1) \dots (r-k+1)}{k!} \quad (r \in \CC, k \in \NN)
$$
are the \emph{generalized binomial coefficients}, we get \eqref{ZSC3} from \eqref{ZSC1} by comparison of the coefficients of these power series. We prove \eqref{ZSC2} by induction. First, \eqref{ZSC1} implies $A_{1} = B_{1}$, and also
$$
Z(t) \prod_{i = 1}^{n - 1}(1 - t^{i})^{B_{i}} =
\prod_{i = n}^{\infty}(1 - t^{i})^{-B_{i}} = 1 + B_{n} t^{n} \ (\mathrm{mod} \ t^{n + 1}).
$$
If $B_{i} \in \ZZ$ for $1 \leq i \leq n - 1$, the left hand side is a power series with coefficients in $\ZZ$, and in particular the coefficient of $t^{n}$. Then $N_{n} \in \ZZ$ by \eqref{Moebius}.
\end{proof}
For $n \in \ZZ$ we define $\pi_{n} = (q^{n + 1} - 1)/(q - 1)$, in such a way that $\pi_{- 1} = 0$ and
$$
\frac{1}{(1 - t) (1 - q t)} = \sum_{n = 0}^{\infty} {\pi_{n} t^{n}}.
$$
The definition of $Z(t)$ implies
\begin{equation}
\label{Newton}
A_{n} = \sum_{k = 0}^{\min(n, 2g)} a_{k} \pi_{n - k}, \quad n \geq 0.
\end{equation}

\begin{lemma}
\label{LMD}
Let $P \in \mathbf{P}_{g}(q)$, and assume $g \geq 2$. If $n \in \ZZ$, then
\begin{equation}
\label{GillesMireille1}
A_{n} = q^{n + 1 - g} A_{2 g - 2 - n} + P(1) \pi_{n - g},
\end{equation}
assuming that $A_{n} = 0$ si $n < 0$. In particular,
\begin{equation}
\label{GillesMireille2}
A_{n} = P(1) \, \pi_{n - g}, \quad n \geq 2g - 1.
\end{equation}
\end{lemma}

\begin{proof}
Since \eqref{GillesMireille1} is invariant by $n \mapsto 2 g - n - 2$, and trivial if $n = g - 1$, it is sufficient to prove \eqref{GillesMireille1} if $n \geq g$. This is performed in three steps. (a) If $n \in \ZZ$, then
\begin{multline*}
(q - 1) \sum_{k = 0}^{2g} a_{k} \pi_{n - k} = 
\sum_{k = 0}^{2g} a_{k} (q^{n - k + 1} - 1) =
q^{n + 1} \sum_{k = 0}^{2g} a_{k} q^{-k} - \sum_{k = 0}^{2g} a_{k}
\\ =
q^{n + 1} P(q^{-1}) - P(1) =
(q^{n + 1 - g} - 1)P(1) =  (q - 1) P(1) \pi_{n - g},
\end{multline*}
since $P(q^{-1}) = q^{- g} P(1)$, hence
$$
\sum_{k = 0}^{2g} a_{k} \pi_{n - k} = P(1) \pi_{n - g}.
$$
This proves \eqref{GillesMireille2}, since the left hand side is equal to $A_{n}$ (if $n = 2 g - 1$, notice that the term with subscript $k = 2 g$ is zero), and $A_{2 g - 2 - n} = 0$, since $2 g - 2 - n \leq - 1$.
(b) If $0 \leq n \leq 2 g - 2$, then
\begin{multline*}
q^{n + 1 - g} A_{2 g - n - 2}
=
\frac{q^{n + 1 - g}}{q - 1}  \sum_{k = 0}^{2g - n - 2} a_{k} (q^{2 g - n - 1 - k} - 1) \\
=
\frac{- 1}{q - 1} \sum_{k = 0}^{2g - n - 2} a_{k} q^{g - k} (q^{n - 2 g + k + 1} - 1)
=
- \sum_{k = 0}^{2g - n - 2} a_{2 g - k} \pi_{n - 2 g + k} = - \sum_{k = n + 2}^{2g} a_{k} \pi_{n - k}.
\end{multline*}
(c) If $0 \leq n \leq 2 g - 2$, we deduce from (a) :
$$
A_{n} = \sum_{k = 0}^{n} a_{k} \pi_{n - k} =
\sum_{k = 0}^{2g} a_{k} \pi_{n - k} - a_{n + 1} \pi_{- 1} - \sum_{k = n + 2}^{2g} a_{k} \pi_{n - k}.
$$
The first term is equal to $P(1) \pi_{n - g}$ by (a), the second is zero, and the third is equal to $q^{n + 1 - g} A_{2 g - n - 2}$ by (b).
\end{proof}

As in the case of the zeta function of a curve \cite{L-MD}, we deduce from Lemma \ref{LMD} that if $g \geq 2$, then
\begin{equation}
\label{Hecke}
\sum_{n = 0}^{g - 1} A_{n} t^{n} + \sum_{n = 0}^{g - 2} q^{g - 1 - n} A_{n} t^{2 g - 2 - n} =
Z_{P}(t) - \frac{P(1) t^{g}}{(1 - t)(1 - q t)} \, .
\end{equation}
In particular, if $t = q^{-1/2}$,
\begin{equation}
\label{Center}
A_{g - 1} + 2 q^{(g - 1)/2} \sum_{n = 0}^{g - 2} A_{n} q^{-n/2} =
q^{(g - 1)/2} Z_{P}(q^{-1/2}) + \frac{P(1)}{(q^{1/2} - 1)^{2}} \, .
\end{equation}

If $P \in \mathbf{P}_{g}(q)$, the map $\omega \mapsto q/\omega$ is a permutation of the inverse roots of $P$, and we order them in a sequence $(\omega_{i})_{1 \leq i \leq 2 g}$ such that $\omega_{g + i} = q/\omega_{i}$ for $1 \leq i \leq g$. We put
$$x_{i} = - (\omega_{i} + \frac{q}{\omega_{i}}), \quad 1 \leq i \leq g.$$
As in $\S$ \ref{subsec_lower}, we define $\eta = \eta(P) \in \QQ$ by
$$
\frac{1}{\eta} = \frac{1}{g} \sum_{i = 1}^{g} \frac{1}{q + 1 + x_{i}} \, .
$$
Since $P(1) \neq 0$, we are sure that $q + 1 + x_{i} \neq 0$. As in the case of the zeta function of a curve \cite{L-MD}, the evaluation at $t = 1$ of \eqref{Hecke} leads to

\begin{theorem}
\label{GillesMireille3}
Assume $g \geq 2$. With notation as above, if $P \in \mathbf{P}_{g}(q)$, then 
$$
\frac{g}{\eta} \, P(1) = \sum_{n = 0}^{g - 1} A_{n} + \sum_{n = 0}^{g - 2} q^{g - 1 - n} A_{n}. \rlap \qed
$$
\end{theorem}

We recall now the \emph{exponential formula}, see R.P. Stanley \cite{sta}, which goes back to Euler. Let $\mathbf{y}=(y_n)_{n\in\NN}$ be a sequence of indeterminates. To an element $b = (b_1, \dots, b_n) \in \NN^n$ one associates the monomial  $\mathbf{y}^b=y_1^{b_1}\dots y_n^{b_n}$ in the ring $\QQ[[\mathbf{y}]]$. Then
\begin{eqnarray*}
\exp \Bigl(\sum_{n=1}^{\infty} y_{n}\frac{t^n}{n}\Bigr) & = &
\prod_{n=1}^{\infty}\exp \Bigl(y_{n}\frac{t^n}{n}\Bigr) =
\prod_{n=1}^{\infty} \sum_{b_n=0}^{\infty}\frac{1}{b_n!}\Bigl(y_n\frac{t^n}{n}\Bigr)^{b_n} \\
& = & \sum_{b_1,\dots b_k\in\NN}\frac{\mathbf{y}^b}{b_1!\dots b_k!}\frac{t^{b_1+2b_2+\dots +kb_k}}{2^{b_2}\dots k^{b_k}}.
\end{eqnarray*}

Let $C_{0}(\mathbf{y}) = 1$ and for $n \in \NN$, $n \geq 1$:
$$C_{n}(\mathbf{y})=\sum_{b \in \EuScript{P}_n} c(b) \mathbf{y}^b, \quad c(b)=\frac{n!}{b_1!\dots b_n!}\frac{1}{2^{b_2}\dots n^{b_n}}$$
where $\EuScript{P}_{n}$ is as above. We put
$$\calC_n(\mathbf{y})=\frac{C_n(\mathbf{y})}{n!}.$$
The previous computations show that the following classical identity, called the \emph{exponential formula}, holds in the ring $\QQ[[\mathbf{y}]][[t]]$:
$$\exp \Bigl(\sum_{n=1}^{+\infty} y_{n}\frac{t^n}{n}\Bigr)=\sum_{n=0}^{+\infty}\calC_n(\mathbf{y})t^n.$$
The coefficients of $C_n(\mathbf{y})$ are natural numbers, since
$$C_n(\mathbf{y})=\sum_{\sigma\in\mathfrak{S}_{n}}\mathbf{y}^{\beta (\sigma )},$$
where the summation is over the symmetric group $\mathfrak{S}_{n}$, where $\beta (\sigma )=(b_1(\sigma ),\dots ,b_n (\sigma))$ and $b_k(\sigma )$ is the number of cycles of length $k$ in the cycle decomposition of  $\sigma$ as a product of disjoint cycles. Now come back to the zeta function $Z_{P}(t)$ of a polynomial $P \in \mathbf{P}_{g}(q)$. From \eqref{expressionzeta}, we deduce
\begin{equation}
\label{An}
A_{n} = \calC_{n}(N_{1}, \dots, N_{n}).
\end{equation}

\subsection{Specific bounds for Jacobians}
\label{subsec_specific}

All the results of $\S$ \ref{subsec_virtual} apply if we take for $P$ the reciprocal Weil polynomial $P_{A}$ of an abelian variety $A$ over $\FF_{q}$, and $\card{A(\FF_q)} = P_{A}(1)$. We put
$$Z_{A}(t) = Z_{P}(t), \quad A_{n}(A) = A_{n}(P), \quadÊB_{n}(A) = B_{n}(P),  \quadÊN_{n}(A) = N_{n}(P)$$
($N_{n}(A)$ should not be confused with $\card{A(\FF_{q^{n}})}$). If $A$ is of dimension $g$ and trace $-\tau$, then
$$
P_{A}(t) = \sum_{n = 0}^{2 g} a_{n} t^{n}, \quad \card{a_{n}} \leq \binom{2 g}{n} q^{n/2}, \quad a_{1} = \tau.
$$
We deduce from \eqref{Newton} that, roughly speaking, if $q$ is large with respect to $g$, then
\begin{equation}
\label{LangWeilSym}
A_{n} = \pi_{n} + \pi_{n - 1} \tau + O(q^{n - 1}) = q^{n} + q^{n - 1} \tau + O(q^{n - 1}),
\end{equation}
where the involved constants depends only on $n$ and $g$. Now consider a curve $C$ over $\FF_{q}$ with Jacobian $J$ as above, and take $P = P_{J}$. Then $Z_{J}(t) = Z_{C}(t)$ is the zeta function of $C$. Moreover, $A_{n}$ and $B_{n}$ are respectively the number of effective and prime rational divisors of degree $n$ of $C$, and $N_{n} = \card{C(\FF_{q^{n}})}$. We record for future use that in this case the following conditions hold:
\begin{equation}
\label{ConditionB}
\tag{\bf{B}}
B_{n} \geq 0 \hskip13mmÊ\text{for} \quad 1 \leq n \leq 2 g,
\end{equation}
\begin{equation}
\label{ConditionN}
\tag{\bf{N}}
N_{n} \geq N_{1} \geq 0 \quad \text{for} \quad 1 \leq n \leq 2 g.
\end{equation}

The reality is that these conditions hold for any $n \geq 1$, and more precisely $N_{n} \geq N_{d}$ if $d |Ên$. The Jacobian of a curve satisfies $\eqref{ConditionB}$ and $\eqref{ConditionN}$, unlike a general abelian variety. Notice that $\eqref{ConditionB}$ is stronger than $\eqref{ConditionN}$, since if $\eqref{ConditionB}$ holds, one deduces from \eqref{Moebius}:
$$n B_{n} \leq N_{n} - N_{1} \quadÊ\text{for} \quad n \geq 2.$$
We discuss below some lower bounds for $N_{n}$ and $B_{n}$ which are valid for any abelian variety. The results show that $\eqref{ConditionB}$ and $\eqref{ConditionN}$ can be peculiar to Jacobians only if $g$ is sufficiently large with respect to $q$.

\begin{lemma}
\label{LargeGenusN}
Let $A$ be an abelian variety of dimension $g \geq 1$ over $\FF_{q}$. If $g \leq q/m$, then $N_{1} = B_{1} \geq  1$. If
$$g \leq \frac{q - q^{1/2}}{2}, \quad \text{then} \quad N_{n} \geq N_{1} \quad \text{for} \ n \geq 1,$$
and, hence, $B_{n} \geq 0$ if $n$ is prime.
\end{lemma}

\begin{proof}
The first statement comes from Serre's inequality. By Weil's inequality, $N_{n} - N_{1} \geq f_{q}(g)$, with
$$f_{q}(g) = q^{n} - q - 2 g (q^{n/2} + q^{1/2}),$$
and $g \mapsto f_{q}(g)$ is decreasing. If $g \leq g_{0} = (q - q^{1/2})/2$, then
$$f_{q}(g) \geq f_{q}(g_{0}) = (q^{n/2} + q^{1/2}) (q^{n/2} - q),$$
hence, $f_{q}(g) \geq 0$ if $n \geq 2$.
\end{proof}

The following two propositions improve some results of N. Elkies and al. \cite[Lem. 2.1(i)]{EHKPWZ}.

\begin{proposition}
\label{IneqPrimeDiv}
Let $A$ be an abelian variety of dimension $g \geq 1$ over $\FF_{q}$. If $n \geq 2$, then
$$
\card{n B_{n} - q^{n}} \leq (2 g + 2) q^{n/2} + 4 g q^{n/4} - (4 g + 2),
$$
and
$$
n B_{n} \geq (q^{n/4} + 1)^{2}((q^{n/4} - 1)^{2} - 2 g).
$$
\end{proposition}

\begin{proof}
Since
$$N_{n} = q^{n} + 1 + \tau_{n}, \quad \text{where} \quad
\tau_{n} = - \sum_{i = 1}^{g}(\omega_{i}^{n}+\bar{\omega}_{i}^{n}),
$$
we find
\begin{eqnarray*}
n B_{n} & = &
\sum_{d | n} \mu(\frac{n}{d}) (q^{d} + 1 + \tau_{d}) =
\sum_{d | n} \mu(\frac{n}{d}) q^{d} + 0 + \sum_{d | n} \mu(\frac{n}{d}) \tau_{d}
\\
\card{n B_{n} - q^{n}} & \leq &
\sum_{d | n, d < n} \card{\mu(\frac{n}{d})}  q^{d} +
\sum_{d | n} \, \card{\mu(\frac{n}{d})} \, \card{\tau_{d}}
\\ & \leq &
\sum_{d | n, d < n} \card{\mu(\frac{n}{d})}  q^{d} + 2 g
\sum_{d | n} q^{d/2}
\leq
2 g q^{n/2} +
\sum_{d | n, d < n} (q^{d} + 2 g q^{d/2}).
\end{eqnarray*}
Since $q/(q - 1) \leq 2$, we have
$$
\sum_{d | n, d < n} q^{d} \leq q + q^{2} + \dots + q^{[n/2]} =
\frac{q^{[n/2] + 1} - 1}{q - 1} - 1 \leq
\frac{q}{q - 1} (q^{n/2} - 1) \leq 2 q^{n/2} - 2,
$$
and
$$
\sum_{d | n, d < n} (q^{d} + 2 g q^{d/2}) \leq 2 q^{n/2} + 4 g q^{n/4} - (4 g + 2).
$$
This implies the first statement, and hence, $n B_{n} \geq F(q^{n/4})$, where
$$
F(x) = x^{4} - (2 g + 2) x^{2} - 4 g x + 4 g + 2 = G(x) + 6 g + 1,
$$
$$
G(x) = (x + 1)^{2}((x - 1)^{2} - 2 g).
$$
Since $F(x) > G(x)$ for any real number $x$, the second statement follows.
\end{proof}

\begin{proposition}
\label{LargeGenusB}
Let $A$ be an abelian variety of dimension $g \geq 2$ over $\FF_{q}$, and let $n \geq 2$.
\begin{enumerate}
\item
\label{LRG1}
If $n > 4 \log_{q}(1 + (2 g)^{1/2})$, that is, if
$$g < \frac{(q^{n/4}- 1)^{2}}{2}, \quad \text{then} \quad B_{n} \geq 1.$$
\item
\label{LRG2}
If
$$g \leq \frac{q - q^{1/2}}{2}, \quad \text{then} \quad B_{n} \geq 0.$$
\item
\label{LRG3}
If $n \geq g$, then $B_{n} \geq 1$, unless $2 \leq g \leq 9$ and $q \leq 5$.
\item
\label{LRG4}
If $n \geq 2 g$, then $B_{n} \geq 1$, unless $2 \leq g \leq 3$ and $q = 2$. If so, $B_{n} \geq 1$ if $n \geq 2 g + 1$.
\end{enumerate}
\end{proposition}

\begin{proof}
The second statement of Proposition \ref{IneqPrimeDiv} implies \eqref{LRG1}, since $B_{n} \in \ZZ$. In the proof of \eqref{LRG2}, we can assume $n \geq 4$, since Lemma \ref{LargeGenusN} implies  $B_{2} \geq 0$ and $B_{3} \geq 0$. But $q - q^{1/2} < (q^{n/4}- 1)^{2}$ if $n \geq 4$, and \eqref{LRG1} implies \eqref{LRG2}. Let
$$
f_{q}(g) = 4 \log_{q}(1 + (2 g)^{1/2}).
$$
Then $q \mapsto f_{q}(g)$ is decreasing, and $g \mapsto f_{q}(g)$ is increasing. If $g \geq 10$, then $g > f_{2}(g)$ and \eqref{LRG1} implies $B_{n} \geq 1$ for every $n \geq g$. On the other hand, if $g \geq 3$ and $q \geq 7$, then $g > f_{7}(g)$ as well. It is easy to see that a ``strict'' version of \eqref{LRG2} holds, namely, if $2 g < q - q^{1/2}$, and in particular if $g = 2$ and $q \geq 7$, then $B_{n} > 0$ for every $n \geq 2$. This proves \eqref{LRG3}.
If $g \geq 4$, then $2 g > f_{2}(g)$ and \eqref{LRG1} implies that $B_{n} \geq 1$ for every $n \geq g$. On the other hand, if  $g \geq 2$ and $q \geq 3$, then $2g > f_{3}(g)$, hence $q = 2$ and $g \leq 3$ if the statement is false, and this proves the first statement of \eqref{LRG4}. If $n \geq 2 g + 2$, then $B_{n} \geq 1$, because, firstly,  $2 g + 2 > f_{2}(g)$ if $g \geq 3$, and secondly, with $F(x)$ as above, and if $g = 2$, then $F(2^{3/2}) = 2(13 - 8 \sqrt{2}) > 0$, and $F(q)$ is increasing if $q \geq 2$. In the two remaining cases, if  $n = 2 g + 1$, we use Serre's inequality: if $g = 2$ then $B_{5} \geq 4$, and if $g = 3$ then $B_{7} \geq 54$. This proves \eqref{LRG4}.
\end{proof}

\begin{example*}
Here are two instances of the exceptional cases appearing in \eqref{LRG4}. Consider the elliptic curves over
$\FF_{2}$:
$$
E_{1} : y^{2} + y = x^{3}, \quad E_{2} : y^{2}  + x y = x^{3} + x^{2} + 1.
$$
Then $f_{E_{1}}(t) = t^{2} + 2$ and $f_{E_{2}}(t) = t^{2} - t + 2$. If $A = E_{1} \times E_{2}$, then $B_{4}(A) = - 1$ and $N_{4}(A) - N_{1}(A) = 6$. If $A = E_{2} \times E_{2} \times E_{2}$, then $B_{6}(A) = 0$  and $N_{6}(A) - N_{1}(A) = 38$.
\end{example*}

We shall use the results of $\S$ \ref{subsec_virtual} in order to obtain lower bounds for the number of rational points on the Jacobian of a curve. To do that, we consider an abelian variety, and assume that Condition \eqref{ConditionB} or Condition \eqref{ConditionN} is satisfied.

\begin{lemma}
\label{minorationAn}
Assume that \eqref{ConditionB} holds. If $n\geq 2$,
$$A_{n} \geq \binom{N_{1} + n - 1}{n}  + \sum_{i=2}^{n} B_{i} \binom{N_{1} + n - i -1 }{n - i}.$$
\end{lemma}

\begin{proof}
All the terms in the right hand side of \eqref{AnFormula} are $\geq 0$. In order to get a lower bound, we sum over the subset of $\EuScript{P}_{n}$ consisting of
$$
(n,0,\ldots,0), (n-2,1,0,\ldots,0), (n-3,0,1,0,\ldots,0), \ldots, (1,0,\ldots,0,1,0), (0,\ldots,0,1).
\mbox{\qedhere}
$$
\end{proof}

From \eqref{GillesMireille2} we deduce
\begin{equation}
\label{FormuleJac1}
\card{A(\FF_q)} = \frac{q - 1}{q^{g} - 1} A_{2 g - 1}.
\end{equation}
From \eqref{FormuleJac1} and Lemma \ref{minorationAn} we deduce:

\begin{proposition}
\label{Mino}
Let $A$ be an abelian variety, and assume that \eqref{ConditionB} holds. Then
\begin{equation}
\label{Borne-gy1}
\tag{\bf{III}}
\card{A(\FF_q)} \geq \frac{q-1}{q^g-1} 
\left[\binom{N+2g-2}{2g-1} + \sum_{i=2}^{2g-1}B_i \binom{N+2g-2-i}{2g-1-i}\right].
\end{equation}
In particular:
\begin{equation}
\label{Borne-gy0}
\card{A(\FF_q)} \geq \frac{q-1}{q^g-1}  \binom{N+2g-2}{2g-1}. \rlap \qed
\end{equation}
\end{proposition}

\begin{remarks*}
(i) Notice that the bound \eqref{Borne-gy1} can be made more explicit, using Proposition \ref{IneqPrimeDiv}.

(ii)
The bound \eqref{Borne-gy0} is not optimal for $g \geq 10$ or $q \geq 7$, by Lemma \ref{LargeGenusB}\eqref{LRG3}.
\end{remarks*}

\begin{lemma}
\label{AnBound}
Let $A$ be an abelian variety, and assume that \eqref{ConditionN} holds. Then
$$A_{n} \geq \binom{N+n-1}{n}, \quad n \geq 1.$$
\end{lemma}

\begin{proof}
We recall two classical results. Let $\mathbf{y}=(y_n)_{n\in\NN}$ and $\mathbf{z}=(z_n)_{n\in\NN}$ be two sequences of indeterminates, and take $n \in \NN$. Firstly, by applying the exponential formula to
$$
\exp \Bigl(\sum_{n=1}^{+\infty} (y_{n}+z_n)\frac{t^n}{n}\Bigr) =
\exp \Bigl(\sum_{n=1}^{+\infty} y_{n}\frac{t^n}{n}\Bigr)
\exp \Bigl(\sum_{n=1}^{+\infty} z_{n}\frac{t^n}{n}\Bigr),
$$
and expanding the right hand side, we obtain:
\begin{equation}
\label{CnSum}
\calC_n(\mathbf{y}+\mathbf{z})   = \sum_{k=0}^n\calC_k(\mathbf{y})\calC_{n-k}(\mathbf{z}).
\end{equation}
Secondly, if $M \in \NN$, then
$$
\exp \Bigl(\sum_{n=1}^{+\infty}M\frac{t^n}{n}\Bigr) = \exp \Bigl(\sum_{n=1}^{+\infty}\frac{t^n}{n}\Bigr)^M=(1-t)^{-M}=\sum_{n=0}^{+\infty}\binom{M+n-1}{n}t^n,
$$
hence,
\begin{equation}
\label{SumCnCst}
\calC_{n}(M,\dots ,M)  =  \binom{M+n-1}{n}.
\end{equation}
We define
$$\mathbf{n} = (N,N,\dots ), \quadÊ\mathbf{d} = (0,N_2 - N ,N_3 - N,\dots ).$$
Then, by \eqref{An}, \eqref{CnSum} and \eqref{SumCnCst}, 
$$
A_{n} = \calC_{n}(\mathbf{n} + \mathbf{d}) =
\sum_{k = 0}^{n} \calC_{k}(\mathbf{n}) \calC_{n - k}(\mathbf{d}) \geq \calC_{n}(\mathbf{n}) = \binom{N+n-1}{n}.
\mbox{\qedhere}
$$
\end{proof}

We put, if $k \geq 2$,
\begin{eqnarray*}
\calX_{k}(N) & = & \binom{N+k-1}{k} -q \binom{N+k-3}{k-2} \\
& = & 
\binom{N + k - 3}{k - 2} \left[ \left( \frac{N - 1}{k} + 1 \right) \left( \frac{N - 1}{k - 1} + 1 \right) - q \right].
\end{eqnarray*}

\begin{lemma}
\label{PreMino}
Let $A$ be an abelian variety, and assume that \eqref{ConditionN} holds. If $k \geq 0$, then $\calC_{k}(\mathbf{n}) \geq 0$, $\calC_{k}(\mathbf{d}) \geq 0$, and
$$
\card{A(\FF_q)} = 
\calC_{g}(\mathbf{d}) + (N - 1) \calC_{g-1}(\mathbf{d}) +
\sum_{k=2}^{g} \calX_{k}(N) \, \calC_{g-k}(\mathbf{d}).
$$
\end{lemma}

\begin{proof}
By Condition \eqref{ConditionN}, the coordinates of  $\mathbf{n}$ and $\mathbf{d}$ are $\geq 0$, and this implies our first assertion. Let $\mathbf{y}$ and $\mathbf{z}$ be two sequences of indeterminates. By \eqref{CnSum},
\begin{eqnarray*}
\calC_g(\mathbf{y}+\mathbf{z}) -q\calC_{g-2}(\mathbf{y}+\mathbf{z})
& = &\sum_{k=0}^g\calC_k(\mathbf{y})\calC_{g-k}(\mathbf{z})-q\sum_{k=0}^{g-2}\calC_k(\mathbf{y})\calC_{g-2-k}(\mathbf{z})\\
& = &\sum_{k=0}^g\calC_k(\mathbf{y})\calC_{g-k}(\mathbf{z})-q\sum_{k=2}^{g}\calC_{k-2}(\mathbf{y})\calC_{g-k}(\mathbf{z})\\
& = & \calC_0(\mathbf{y})\calC_{g}(\mathbf{z})+\calC_1(\mathbf{y})\calC_{g-1}(\mathbf{z})+ \sum_{k=2}^g(\calC_k(\mathbf{y})-q\calC_{k-2}(\mathbf{y}))\calC_{g-k}(\mathbf{z}).
\end{eqnarray*}
Now, by applying \eqref{GillesMireille1} with $n = 0$, we get
\begin{equation}
\label{FormuleJac3}
\card{A(\FF_q)} = A_g - q A_{g-2},
\end{equation}
and using \eqref{An},
$$
\card{A(\FF_q)} = \calC_g(\mathbf{n}+\mathbf{d}) -q\calC_{g-2}(\mathbf{n}+\mathbf{d}).
$$
Replacing $\mathbf{y}$ by $\mathbf{n}$ and $\mathbf{z}$ by $\mathbf{d}$, and since \eqref{SumCnCst} implies
$$\calC_{k}(\mathbf{n}) = \binom{N+k-1}{k},$$
we get the required expression for $\card{A(\FF_q)}$. 
\end{proof}

\begin{proposition}
\label{Mino2}
Assume $g \geq 2$. Let $A$ be an abelian variety. Assume that \eqref{ConditionN} holds and that
\begin{equation}
\label{condBorne}
\left(\frac{N-1}{g}+1\right)\left(\frac{N-1}{g-1}+1\right)-q > 0.
\end{equation}
Then
\begin{equation}
\label{Borne-sh}
\tag{\textbf{IV}}
\card{A(\FF_q)} \geq \binom{N+g-1}{g} -q\binom{N+g-3}{g-2}.
\end{equation}
\end{proposition}

\begin{proof}
The right hand side of \eqref{Borne-sh} is equal to $\calX_{g}(N)$, and $\calX_{g}(N) > 0$ if and only if \eqref{condBorne} holds, in which case $N \geq 1$. For $k = 2,\dots, g$, we have
$$
\calX_{k}(N) \geq \binom{N+k-3}{k-2}\left(\left(\frac{N-1}{g}+1\right)\left(\frac{N-1}{g-1}+1\right)-q\right) \geqÊ0,
$$
where the second inequality comes from  (\ref{condBorne}). Applying Lemma \ref{PreMino} we deduce
$$
\card{A(\FF_q)} \geq  \calC_0(\mathbf{n})\calC_{g}(\mathbf{d})+\calC_1(\mathbf{n})\calC_{g-1}(\mathbf{d}) + \calX_{g} \, \calC_{0}(\mathbf{d}) \geq \calX_{g} \,\calC_{0}(\mathbf{d}),
$$
and the result follows, since $\calC_{0}(\mathbf{d}) = 1$.
\end{proof}

\begin{remarks*}
(i) The condition \eqref{condBorne} is satisfied if $N \geq g (q^{1/2} - 1) + 1$. This inequality has to be compared to the Drinfeld-Vl\u{a}du\c{t} upper bound.

(ii) Notice that Proposition \ref{Mino2} can be improved: since
$$C_n(\mathbf{d})=\sum_{b\in\EuScript{P}_n}c(b)\mathbf{d}^b \geq \frac{N_n-N}{n} \,,$$
because the right hand side is the term of the sum corresponding to $b=(0,\dots , 0,1)$, we get
$$ \calC_0(\mathbf{n})\calC_{g}(\mathbf{d})\geq \frac{N_g-N}{g}\quad\mbox{ and }\quad\calC_1(\mathbf{n})\calC_{g-1}(\mathbf{d})\geq N\frac{N_{g-1}-N}{g-1}.$$
Therefore if (\ref{condBorne}) holds, then
$$\card{J_C(\FF_q)} \geq \frac{N_{g}-N}{g}+N\frac{N_{g-1}-N}{g-1}+\binom{N+g-1}{g} -q\binom{N+g-3}{g-2},$$
and the numbers $N_{g}$ and $N_{g - 1}$ can be replaced by their standard lower bounds in order to get a bound improving \eqref{Borne-sh}.

(iii)
Lemma \ref{LMD} provides some others identities than \eqref{FormuleJac1} and \eqref{FormuleJac3}, for instance
$$
A_{2 g - 2} = \card{A(\FF_q)} \pi_{g - 2} + q^{g - 1}.
$$
On the other hand, $Z_{A}(t) < 0$ if $q^{-1} < t < 1$, because $f_{A}(t) \geq 0$ for any $t \in \RR$, and one deduces from \eqref{Center} the inequality
$$
A_{g - 1} \leq \dfrac{\card{A(\FF_q)}}{(q^{1/2} - 1)^{2}} - 2 q^{(g - 1)/2},
$$
as established by S. Ballet, C. Ritzenthaler and R. Rolland in \cite{BRR}. These relations lead to lower bounds similar to \eqref{Borne-gy1} and \eqref{Borne-sh}.
\end{remarks*}

Assume $g \geq 2$. From Theorem \ref{GillesMireille3}, we know that
\begin{equation}
\label{LMDsomme}
\frac{g}{\eta} \, \card{A(\FF_q)} = \sum_{n = 0}^{g - 1} A_{n} + \sum_{n = 0}^{g - 2} q^{g - 1 - n} A_{n},
\end{equation}
where $\eta = \eta(A)$ is the harmonic mean of the numbers $q + 1 + x_{i}$, as in $\S$ \ref{subsec_lower}. We recall from \cite{L-MD} that
\begin{eqnarray}
\label{LMDsigma1}
\eta & \geq & (q^{1/2} - 1)^2, \\
\label{LMDsigma2}
\eta & \geq & \frac{g (q-1)^2}{(g+1)(q+1)-N},
\end{eqnarray}
and \eqref{LMDsigma2} is always tighter than \eqref{LMDsigma1}, though it does depend on $N$.
Moreover, if $q \geq 8$, by Proposition \ref{Harmonic}, we know that
$$
\eta \geq q + 1 - m .
$$
This lower bound is better than the uniform lower bound deduced from \eqref{LMDsigma2}, namely
$$
\eta \geq \frac{g (q-1)^2}{(g+1)(q+1)-N} \geq
\frac{(q - 1)^2}{q + 1 + m} = q + 1 - m - \frac{4 q - m^{2}}{q + 1 + m} \, .
$$

\begin{theorem}
\label{gy2}
If  $g\geq 2$, and if \eqref{ConditionN} holds, then
\begin{equation}
\label{Borne-gy2}
\tag{\textbf{V}}
\card{A(\FF_q)} \geq \frac{\eta}{g} \left[ 
\binom{N + g - 2}{g - 2} +
\sum_{n=0}^{g-1}q^{g-1-n} \binom{N+n-1}{n} \right]
\end{equation}
\end{theorem}

\begin{proof}
Since \eqref{ConditionN} holds, we apply the inequality of Lemma \ref{AnBound} in \eqref{LMDsomme}. Noticing that
$$
\sum_{n = 0}^{g - 2} \binom{N + n - 1}{n} = \binom{N + g - 2}{g - 2},
$$
we obtain the result.
\end{proof}

\begin{remark*}
The expression in brackets is $\geq q^{g - 1}$. Since $N \geq 0$, we recover as a corollary the bound \cite[Th. 2(1)]{L-MD}, which does not depend on $N$:
$$\card{J_{C}(\FF_q)} \geq q^{g-1}\frac{(q-1)^2}{(q+1)(g+1)}.$$
\end{remark*}

The right hand side of \eqref{Borne-gy2} is cumbersome. Here is a simpler lower bound using the partial sums of the exponential series. Let
$$
e_n(x)=\sum_{j=0}^n\frac{x^j}{j!}, \quad n \in \NN, \quad x>0.
$$
Then
$$
e_n(x)=e^x\frac{\Gamma(n+1,x)}{n!}, \quad \text{where} \quad 
\Gamma(n,x)=\int_{x}^{\infty}t^{n-1}e^{-t}dt
$$
is the incomplete Gamma function. Since
$$\binom{N+n-1}{n}\geq \frac{N^n}{n!},$$
we get from  Theorem \ref{gy2}:

\begin{corollary}
If $g\geq 2$, then
$$
\card{A(\FF_q)} \geq
\left[ \binom{N + g - 2}{g - 2} + q^{g-1}e_{g-1}(q^{-1}N) \right] \frac{(q-1)^2}{(g+1)(q+1)-N}.
\rlap \qed
$$
\end{corollary}

\subsection{Discussing the bounds}
\label{subsec_discussion}

Let $C$ be a curve of genus $g \geq 2$ over ${\mathbb F}_q$, with $N = \card{C(\FF_q)}$. We recall the lower bounds for the number of rational points on $J = J_{C}$, respectively obtained from Proposition \ref{BorneSpechtJac}  (with $M(q)$ as defined there), Proposition \ref{firstbound} (with $r$ and $s$ as defined there), Proposition \ref{Mino}, Proposition \ref{Mino2}, and Theorem \ref{gy2}:  
\begin{subeqnarray*}
\eqref{bsj} &
\card{J(\FF_q)} & \geq M(q)^{g} \Bigl(q + 1 + \frac{N - (q + 1)}{g} \Bigr)^{g} \\\\
\eqref{Borne-per} &
\card{J(\FF_q)} & \geq (N -2(r-s) q^{1/2})(q+1+2 q^{1/2})^r(q+1-2 q^{1/2})^s \\\\
\eqref{Borne-gy1} &
\card{J(\FF_q)} & \geq \frac{q-1}{q^g-1} 
\left[\binom{N+2g-2}{2g-1} + \sum_{i=2}^{2g-1}B_i \binom{N+2g-2-i}{2g-1-i}\right] \\\\
\eqref{Borne-sh}  &
\card{J(\FF_q)} & \geq \binom{N+g-1}{g} -q\binom{N+g-3}{g-2} \\\\
\eqref{Borne-gy2} &
\card{J(\FF_q)} & \geq \frac{\eta}{g}
\left[ \binom{N + g - 2}{g - 2}  +\sum_{n=0}^{g-1}q^{g-1-n} \binom{N+n-1}{n}\right] \\
\end{subeqnarray*}

(i) When $q$ is large with respect to $g$, we have, for any abelian variety of dimension $g$,
$$\card{A(\FF_q)} = q^{g} + O(q^{g - \frac{1}{2}}),$$
and \eqref{bsj} and \eqref{Borne-per} are the only bounds to be consistent with this estimate. More precisely, since \eqref{Borne-per} is usually reached for abelian varieties when $q$ is a square, this bound is probably the best one as soon as $g\leq(q- q^{1/2})/2$.

\medskip

(ii) 
Assume that \eqref{ConditionN} holds. Then \eqref{LMDsomme}, joint to the inequality $A_{n} \geq N$ for $n \geq 1$, gives
$$
\frac{g}{\eta} \, \card{A(\FF_q)}  \geq (q^{g - 1} - 1) \frac{N + q - 1}{q - 1} \, .
$$
Using \eqref{LMDsigma1}, we recover \cite[Th. 2(2)]{L-MD}:
\begin{equation}
\label{Borne-lmd}
\card{A(\FF_q)}  \geq  (q^{1/2} - 1)^{2} \frac{q^{g - 1} - 1}{g}\frac{N + q - 1}{q - 1} \, .
\end{equation}
But if $n \geq 1$ and $N \geq 1$, the inequality in Lemma \ref{AnBound} is better than $A_{n}\geq N$, since
$$
\binom{N+n-1}{n} \geq N.
$$
Hence, \eqref{Borne-gy2} is always better than \eqref{Borne-lmd}.

\medskip

(iii) The tables in \cite{L-MD} provide numerical evidence that these bounds can be better than those which hold for general abelian varieties, at least when $q$ is not too large.

\medskip

(iv) Assume $q \geq 4$, and let $A$ be any abelian variety of dimension $g \leq (q - q^{1/2})/2$ with $N_{1} \geq 0$. Then the bounds \eqref{Borne-sh} and \eqref{Borne-gy2} hold for $A$, by Lemma \ref{LargeGenusN}. Likewise, the bound \eqref{Borne-gy1} holds for $A$, by Proposition \ref{LargeGenusB}\eqref{LRG2}.

\medskip

(v) The numerical experiments that we performed lead to the following observations. The bound \eqref{Borne-sh} can be good, even if $g\geq 9$, but, when $g$ is large, \eqref{Borne-gy2} seems to be better than \eqref{Borne-gy1} and \eqref{Borne-sh}, and probably \eqref{Borne-gy2} becomes better than  \eqref{Borne-per} when $g$ is very large.

\section{Jacobian surfaces}
\label{sec_jacsurf}

The characteristic polynomial of an elliptic curve determines the number of its rational points, and vice versa. Therefore, the values of $J_q(1)$ and $j_q(1)$ are given by the Deuring-Waterhouse Theorem (see \cite{deur}, \cite{water}): if $q=p^n$, then
$$
J_{q}(1) = \left\{
\begin{array}{lll}
q + 1 + m & \text{if } n = 1, n \ \text{is even, or } p \not \vert \, m,\\
q + m     & \text{otherwise},
\end{array}
\right.
$$
$$
j_{q}(1) = \left\{
\begin{array}{lll}
q + 1 - m & \text{if } n = 1, n \ \text{is even, or } p \not \vert \, m,\\
q + 2 - m & \text{otherwise}.
\end{array}
\right.
$$
The description of the  set of characteristic polynomials of abelian surfaces has been given by H.G.~R\"uck in \cite{ruck}. The question of describing the set of isogeny classes  of abelian surfaces which contain a Jacobian has been widely studied, especially by J.-P. Serre \cite{serre0}, \cite{serre}, \cite{serre2}, whose aim was to determine $N_q(2)$. A complete answer to this question was finally given by  E. Howe,  E. Nart, and  C. Ritzenthaler in \cite{hnr}. In the remainder of this section, we explain how to deduce from these results the value of $J_q(2)$ and $j_q(2)$. Let $A$ be an abelian surface over $ \FF_q$ of type $[x_1,x_2]$.
Its characteristic polynomial has the form
$$f_A(t)=t^4+a_1t^3+a_2t^2+qa_1t+q^2,$$
with
$$ a_1=x_1+x_2 \quad \text{and} \quad a_2=x_1x_2+2q.$$
By elementary computations, H.G. R\"uck \cite{ruck} showed that  the fact that the roots of $f_A(t)$ are $q$-Weil numbers (i.e. algebraic integers such that their images under every complex embedding have absolute value $q^{1/2}$) is equivalent to
\begin{equation}
\vert a_1\vert\leq 2m \quad \text{and} \quad 2\vert a_1\vert q^{1/2}-2q\leq a_2\leq\frac{a_1^2}{4}+2q. \label{a12}
\end{equation}
We have
\begin{equation}
\card{A(\FF_q)} =f_A(1)=q^2+1+(q+1)a_1+a_2.\label{abeliansurface}
\end{equation}
Table  \ref{tabmax} gives all the possibilities for $(a_1,a_2)$ such that $a_1\geq 2m-2$. Here
$$\varphi_{1} = (- 1 + \sqrt{5})/2, \quadÊ\varphi_{2} = (- 1 - \sqrt{5})/2.$$

\begin{table}[htbp]
\begin{center}
\renewcommand{\arraystretch}{1.5}
\begin{tabular}{| l l | l | l |}
\hline
$a_1$ & $a_2$ & Type & $\card{A(\FF_q)} $\\
\hline
\hline
$2m$ & $m^2+2q$ & $[m,m]$ & $b^2$ \\
\hline
$2m-1$ & $m^2-m+2q$ & $[m,m-1]$ & $b(b - 1)$ \\
 & $m^2-m-1+2q$ & $[m + \varphi_{1},m + \varphi_{2}]$ & $b^2 - b - 1$ \\
\hline 
$2m-2$ & $m^2-2m+1+2q$ & $[m-1,m-1]$ & $(b - 1)^2$ \\
 & $m^2-2m+2q$ & $[m,m-2]$ & $b(b - 2)$ \\
 & $m^2-2m-1+2q$ & $[m-1+\sqrt{2},m-1-\sqrt{2}]$ & $(b - 1)^2 - 2$ \\
 & $m^2-2m-2+2q$ & $[m-1+\sqrt{3},m-1-\sqrt{3}]$ & $(b - 1)^2 - 3$ \\
\hline
\end{tabular}
\end{center}
\caption{Couples $(a_1,a_2)$ maximizing $\card{A(\FF_q)} $, with $b = q + 1 + m$.}
\label{tabmax}
\end{table}
The numbers of points are classified in decreasing order and an abelian variety with $(a_1,a_2)$ not in the table has a number of points strictly less than the values of the table. Indeed, if $-2m\leq a_1< 2m-2$, then
\begin{eqnarray*}
(q+1)a_1+a_2 & \leq & [(q+1)a_1+\frac{{a_1}^2}{4}+2q]\\
& \leq & [(q+1)(2m-3)+\frac{(2m-3)^2}{4}+2q]\\
& = & (q+1)(2m-2)+(m^2-2m-2+2q)+(3-(q+m))\\
& < & (q+1)(2m-2)+(m^2-2m-2+2q)
\end{eqnarray*}
(notice that the function $x\mapsto (q+1)x + (x^{2}/4)$ is increasing on the interval $[-2m,2m-3]$).

In the same way, we build the table of couples $(a_1,a_2)$ with $a_1 \leq - 2 m + 2$. Notice that the ends of the interval containing $a_2$ given by (\ref{a12}) depend only on the value of $a_1$, hence the possible entries for $a_2$ are the same as in the previous table. Here again, the numbers of points are classified in increasing order and an abelian variety with $(a_1,a_2)$ not in the following table has a number of points strictly greater than the values of the table. Indeed, if $-2m+2< a_1\leq 2m$, then
\begin{table}[htbp]
\begin{center}
\renewcommand{\arraystretch}{1.5}
\begin{tabular}{| l l | l | l |}
\hline
$a_1$ & $a_2$ & Type & $A(\FF_q)$ \\
\hline
\hline
$-2m$ & $m^2+2q$ & $[-m,-m]$ & $b'^2$ \\
\hline
$-2m+1$ & $m^2-m-1+2q$ & $[-m+\varphi_{1},-m+\varphi_{2}]$ & $b'^2 - b' - 1$ \\
 & $m^2-m+2q$ & $[-m,-m+1]$ & $b'(b' + 1)$ \\
\hline 
$-2m+2$ & $m^2-2m-2+2q$ & $[-m+1+\sqrt{3},-m+1-\sqrt{3}]$ & $(b' + 1)^2 - 3$ \\
 & $m^2-2m-1+2q$ & $[-m+1+\sqrt{2},-m+1-\sqrt{2}]$ & $(b' + 1)^2 - 2$ \\
 & $m^2-2m+2q$ & $[-m,-m+2]$ & $b'(b' + 2)$ \\
 & $m^2-2m+1+2q$ & $[-m+1,-m+1]$ & $(b' + 1)^2$ \\
\hline
\end{tabular}
\end{center}
\caption {Couples $(a_1,a_2)$ minimizing $\card{A(\FF_q)} $, with $b' = q + 1 - m$.}
\label{tabmin}
\end{table}
\begin{eqnarray*}
(q+1) a_1 + a_2 & \geq & (q + 1) a_1 + 2 \vert a_1 \vert q^{1/2}-2q\\
& \geq & (q+1)(-2m+3)+2(2m-3)q^{1/2}-2q\\
& = & (q+1)(-2m+2)+(m^2-2m+1+2q)
-(2q^{1/2} - m + 1)^2+(q^{1/2}-1)^2\\
& > & (q+1)(-2m+2)+(m^2-2m+1+2q)
\end{eqnarray*}
(notice that the function $x\mapsto (q+1) x + 2\vert x \vert q^{1/2}$ is increasing on the interval $[-2m+3,2m]$).

Most cases of Theorem \ref{Jq2} and \ref{jq2} will be proved in the following way:

\begin{enumerate}
\item Look at the highest row of Table \ref{tabmax} or \ref{tabmin} (depending on the proposition being proved).
\item Check if the corresponding polynomial is the characteristic polynomial of an abelian variety.
\item When it is the case, check if this abelian variety is isogenous to a Jacobian  variety.
\item When it is not the case, look at the following row and come back to the second step.
\end{enumerate}

For the second step, we use the results of H.-G. R\"uck \cite{ruck}  who solved the problem of describing characteristic polynomials of abelian surfaces, in particular the fact that if $(a_1,a_2)$ satisfy (\ref{a12}) and $p$ does not divide $a_2$ then the corresponding polynomial is the characteristic polynomial of an abelian surface. 

For the third step, we use \cite{hnr} where we can find a characterization of isogeny classes of abelian surfaces containing a Jacobian.

The determination of $J_q(2)$ in Theorem \ref{Jq2} is closely related to that of $N_q(2)$, as done by J.-P. Serre \cite{serre2}. In order to simplify the proof of Theorem \ref{jq2}, we use the fact that, given a curve of genus $2$, if we denote by $(a_1,a_2)$ the coefficients associated to its characteristic polynomial, there exists a curve (its quadratic twist) whose coefficients are $(-a_1,a_2)$. This allows us to adapt the proof of Theorem \ref{Jq2}.

Let us recall the definition of special numbers introduced by J.-P. Serre. An odd power $q$ of a prime number $p$ is \textit{special} if one of the following conditions is satisfied (recall that $m=[2 q^{1/2}]$):

\begin{enumerate}
\item $m$ is divisible by $p$,
\item there exists $x\in\ZZ$ such that $q=x^2+1$,
\item there exists $x\in\ZZ$ such that $q=x^2+x+1$,
\item there exists $x\in\ZZ$ such that $q=x^2+x+2$.
\end{enumerate}

\begin{remark*}
In \cite{serre}, J.-P. Serre asserts that if $q$ is prime then the only possible conditions are conditions (2) and (3). When $q$ is not prime, then condition (2) is impossible,  condition (3) is possible only if $q=7^3$ and condition (4) is possible only if $q=2^3$, $2^5$ or $2^{13}$. Moreover, using basic arithmetic, it can be shown (see \cite{lau1} for more details) that conditions (2), (3) and (4) are respectively equivalent to $m^2-4q=-4$, $-3$ and $-7$.
\end{remark*}

\begin{theorem}
\label{Jq2}
The complete set of values of $J_q(2)$ is given by the following display.
\begin{enumerate}
\renewcommand{\theenumi}{\alph{enumi}}
\item
\label{Jq2sq}
Assume that $q$ is a square. Then
$$
J_{q}(2) = \left\{
\begin{array}{lll}
(q + 1 + m)^{2} & \text{if} & q \neq 4,9.\\
55              & \text{if} & q = 4.\\
225             & \text{if} & q = 9.
\end{array}
\right.
$$
\item
\label{Jq2notsq}
Assume that $q$ is not a square. If $q$ is not special, then
$$
J_{q}(2) = (q + 1 + m)^{2}.
$$
If $q$ is special, then
$$
J_{q}(2) = \left\{
\begin{array}{lll}
(q + 1 + m + \varphi_{1})(q + 1 + m + \varphi_{2}) & \text{if} & \{2  q^{1/2}\} \geq \varphi_{1}. \\
(q + m)^2 & \text{if} & \{2  q^{1/2}\} < \varphi_{1}, p \neq 2 \ \text{or} \ p \vert m.\\
(q+1+m)(q-1+m) && \text{otherwise}.
\end{array}
\right.
$$
Here $\varphi_{1} = (- 1 + \sqrt{5})/2, \varphi_{2} = (- 1 - \sqrt{5})/2$.
\end{enumerate}
\end{theorem}

\begin{proof}
\eqref{Jq2sq} Assume that $q$ is a square.

\noindent ---
If $q\neq 4,9$,  $N_q(2)$  is the Serre-Weil bound \cite{serre}, thus there exists a curve of  type $[m,m]$.

\noindent ---
If $q=4$, then $m=4$. First we prove that $J_4(2)\leq 55$. Every curve of genus  $2$ over $\FF_q$ is hyperelliptic, therefore, the number of rational points is at most $2(q+1)=10$. We deduce that a Jacobian of dimension $2$ over  $\FF_4$ must have $a_1\leq 10-(q+1)=5$. 

If $a_1=5$ then $a_2\leq 14$ by \eqref{a12}. An abelian surface over ${\mathbb F}_4$ with $(a_1,a_2)=(5,14)$ is of type $[3,2]$ and is never  a Jacobian  (because $x_1-x_2=3-2=1$, see \cite{hnr}). Thus we have $a_2\leq 13$ and a Jacobian surface over ${\mathbb F}_4$ with $a_1=5$ has at most  $q^2+1+5(q+1)+13=55$ points. If $a_1<5$, then
$$
q^2+1+(q+1)a_1+a_2 \leq q^2+1+(q+1)a_1+\frac{{a_1}^2}{4}+2q \leq 49
$$
(notice that the function $x\mapsto 5x+ (x^{2}/4)$ is increasing on $[-8,4]$, and $a_1\geq -8$). Thus  an abelian surface over ${\mathbb F}_4$ with $a_1<5$ has less than  $55$ points, hence $J_4(2)\leq 55$.

It remains to prove that $J_4(2)\geq 55$. An abelian surface over ${\mathbb F}_4$ with $(a_1,a_2)=(5,13)$ is of type $[3 + \varphi_{1},3 + \varphi_{2}]$. Such an abelian surface exists (because $p=2$ does not divide $13$) and by \cite{hnr} it is isogenous to a Jacobian. This Jacobian has  $q^2+1+5(q+1)+13=55$ points. 

\noindent ---
If $q=9$, then $m=6$. Since $2(q+1)=20$, we must have  $a_1\leq 20-(q+1)=10=2m-2$. The highest row of Table \ref{tabmax} such that $a_1 = 2m-2$ is that with type $[m-1,m-1]$, and this is the type of some Jacobian with $(q+m)^2=225$ points.

\medskip
\eqref{Jq2notsq} Assume that $q$ is not a square. This part of the proof follows easily from Serre's results. He proved in \cite{serre2} the following facts:

\noindent ---
There exists a Jacobian of type $[m,m]$ if and only if $q$ is not special.

\noindent ---
An abelian surface of type $[m,m-1]$ is never a Jacobian.

\noindent ---
If $q$ is special, then there exists a Jacobian of type  $[m + \varphi_{1},m + \varphi_{2}]$ if and only if  $\{ 2 q^{1/2}\}\geq \varphi_{1}$. Note that $\{ 2 q^{1/2}\}\geq \varphi_{1}$ is equivalent to $m + \varphi_{1} \leq 2 q^{1/2}$, thus it is obvious that this condition is necessary.

\noindent ---
If $q$ is special, $\{ 2 q^{1/2}\} < \varphi_{1}$, $p\neq 2$ or $p \vert m$, then there exists a Jacobian of type $[m-1,m-1]$.

\noindent ---
If $q$ is special, $\{ 2 q^{1/2}\} < \varphi_{1}$, $p=2$ and $p \nmid m$, that is, $q=2^5$ or $2^{13}$ (if $q=2^3$, then $\{ 2 q^{1/2}\} \geq \varphi_{1}$), then there exists a Jacobian of type $[m,m-2]$.

It remains to prove that for $q=2^5$ and $2^{13}$, there does not exist a Jacobian of type $[m-1,m-1]$. In fact, when $q=2^5$ and $2^{13}$, an abelian variety with all $x_i$ equal to $(m-1)$ must have a dimension respectively multiple of $5$ and $13$ (see \cite{mana}, Prop. 2.5).
\end{proof}

\begin{theorem}
\label{jq2}
The complete set of values of $j_q(2)$ is given by the following display.
\begin{enumerate}
\renewcommand{\theenumi}{\alph{enumi}} 
\item
\label{jq2sq}
Assume that $q$ is a square. Then
$$
j_{q}(2) = \left\{
\begin{array}{lll}
(q + 1 - m)^{2} & \text{if} & q \neq 4,9.\\
5               & \text{if} & q = 4.\\
25              & \text{if} & q = 9.
\end{array}
\right.
$$
\item
\label{jq2notsq}
Assume that $q$ is not a square. If $q$ is not special, then
$$
j_{q}(2) = (q + 1 - m)^{2}.
$$
If $q$ is special, then
$$
j_{q}(2) = \left\{
\begin{array}{lll}
(q + 1 - m - \varphi_{1})(q + 1 - m - \varphi_{2}) & \text{if} & \{2  q^{1/2}\} \geq \varphi_{1}. \\
(q + 2 - m + \sqrt{2})(q + 2 - m - \sqrt{2}) & \text{if} & \sqrt{2} - 1 \leq \{2  q^{1/2}\} < \varphi_{1}. \\
(q + 1 - m) (q + 3 - m) & \text{if} & \{2 q^{1/2}\} < \sqrt{2} - 1, p \not \vert m \ \text{and} \ q \neq 7^{3}.\\
(q + 2 - m)^{2} && \text{otherwise}.
\end{array}
\right.
$$
\end{enumerate}
\end{theorem}

\begin{proof}

\eqref{jq2sq} Assume that $q$ is a square.

---
If $q\neq 4,9$, we saw that there exists a curve of type $[m,m]$, and its quadratic twist is of  type $[-m,-m]$. 

---
If $q=4$, then $m=4$. First we prove that $j_4(2)\geq 5$. We have $a_1\geq -5$ since the quadratic twist of a curve with $a_1<-5$ would have $a_1>5$ and we saw that it is not possible.

If $a_1=-5$ then $a_2\geq 12$ by (\ref{a12}).  An abelian surface over ${\mathbb F}_4$ with $(a_1,a_2)=(-5,12)$ is of type $[-4,1]$ and is never  a Jacobian. Thus $a_2\geq 13$ and a Jacobian surface over ${\mathbb F}_4$ with $a_1=-5$ has at least $q^2+1-5(q+1)+13=5$ points. If $a_1 > -5$, then
$$
q^2+1+(q+1)a_1+a_2 \geq q^2+1+(q+1)a_1+2\vert a_1\vert q^{1/2}-2q = 9 + 5 a_1+4 \vert a_1 \vert \geq 5
$$
(note that the function $x\mapsto 5x+4\vert x \vert$ is increasing on $[-4,8]$). Thus an abelian surface over ${\mathbb F}_4$ with $a_1>-5$ has more than  $5$ points, hence $j_4(2)\geq 5$.

It remains to prove that $j_4(2)\leq 5$. There exists a curve with $(a_1,a_2)=(-5,13)$: the quadratic twist of the curve with $(a_1,a_2)=(5,13)$ in the proof of Theorem \ref{Jq2}. The number of points on its Jacobian is
$$q^2+1-5(q+1)+13=5.$$

--- If $q=9$, then $m=6$. Using the same argument as in the last step, we must have $a_1 \geq -2m+2$. We look at the rows of Table \ref{tabmin}, beginning by the rows on the top, for which $a_1 = -2m+2$. The first two can be ignored since $\{ 2 q^{1/2} \} =0$ is less than $\sqrt 3 -1$ and less than $\sqrt 2-1$. An abelian surface of type $[-m,-m+2]$  is not  a Jacobian (this is an almost ordinary abelian surface, $m^2=4q$ and $m-(m-2)$ is squarefree, see \cite{hnr}). The product  of two copies of an elliptic curve of trace $(m-1)$  is isogenous to a Jacobian (such a curve exists since $3\not\vert (m-1)$).

\medskip
\eqref{jq2notsq} Assume that $q$ is not a square. Using twisting arguments and the proof of Theorem \ref{Jq2}, we see that:

\noindent ---
There exists a Jacobian of type $[-m,-m]$ if and only if $q$ is not special.

\noindent ---
If $q$ is special, there exists a Jacobian of type  $[- m - \varphi_{1},- m - \varphi_{2}]$ if and only if  $\{ 2 q^{1/2} \}\geq \varphi_{1}$.

\noindent ---
An abelian surface of type $[-m,-m+1]$ is never a Jacobian.

In the remainder of the proof, we suppose that $q$ is special and $\{ 2 q^{1/2} \} < \varphi_{1}$.

\noindent ---
In order to have the existence of an abelian surface of type $[-m+1+\sqrt{3},-m+1-\sqrt{3}]$, it is necessary to have  $\{ 2 q^{1/2} \} \geq\sqrt 3 -1$. When $\{ 2 q^{1/2} \} < \varphi_{1}$, this condition is never satisfied  (since $\varphi_{1} < \sqrt 3 -1$).

\noindent ---
In order to ensure the existence of an abelian surface of type $[-m+1+\sqrt{2},-m+1-\sqrt{2}]$,   it is necessary to have   $\{ 2 q^{1/2} \} \geq\sqrt 2 -1$. Suppose that this condition holds. we shall show that there exists an abelian surface of type $[-m+1+\sqrt{2},-m+1-\sqrt{2}]$. We use the same kind of argument that J.-P. Serre used in \cite{serre2}. If $p\vert m$, we are done since $p\not\vert a_2=m^2-2m-1+2q$. Otherwise, $(m-2 q^{1/2} )(m+2 q^{1/2} )=m^2-4q\in\{ -3,-4,-7\}$, hence
$$\{ 2 q^{1/2} \} =2 q^{1/2}  -m=\frac{4q-m^2}{m+2 q^{1/2} }\leq\frac{7}{2m},$$ and if $m\geq 9$, $\frac{7}{2m}<\sqrt{2}-1$.  It remains to consider by hand the powers of primes of the form $x^2+1$, $x^2+x+1$ and $x^2+x+2$ with $m<9$ (i.e. $q<21$). These prime powers are precisely $2$, $3$, $4$, $5$, $7$, $8$, $13$ and $17$. If $q=2,8$, then $\{ 2 q^{1/2} \}\geq \varphi_{1}$. If $q=3$, then $p\vert m$. If $q=4,7,13,17$, then $\{ 2 q^{1/2} \} <\sqrt{2}-1$. If $q=5$, then $m=4$ and $p=5$ do not divide $a_2=m^2-2m-1+2q=17$, and we are done. Finally, using \cite{hnr}, we conclude that this abelian surface is isogenous to a Jacobian.

--- If $\{ 2 q^{1/2} \} <\sqrt 2-1 $, $p\not\vert m$ and $q\neq 7^3$, then $p \not \vert \, (m-2)$. To see this, take $p\neq 2$ (if $p=2$, this is obvious) and use the remark about special numbers in this section. Suppose that $p$ divides $(m-2)$, then $p$ also divides $m^2-4-4q=(m+2)(m-2)-4q$. Since $p\neq 2$, we must have $m^2-4q\in\{-3,-4\}$. If $m^2-4q=-3$, $p$ divides $-3-4=-7$ thus $p=7$. But $q$ is not prime (since for $q=7$, $p\not\vert (m-2)=5$), therefore we must have $q=7^3$ and this case is excluded. If $m^2-4q=-4$, $p$ divides $-4-4=-8$, thus $p=2$ which contradicts our assumption. This proves our assertion, and therefore, there exist elliptic curves of trace $m$ and $(m-2)$ and by \cite{hnr} their product is isogenous to a Jacobian.

--- Suppose that $\{ 2 q^{1/2} \} <\sqrt 2-1 $ and $p\vert m$, or $q=7^3$.  By \cite{water}, if $p\vert m$, there does not exist an elliptic curve of trace $m$ ($q=2$ and $3$ are excluded since in those cases, $\{ 2 q^{1/2} \} \geq\sqrt 2-1 $). If $q=7^3$ (thus $(m-2)=35$) there does not exist an elliptic curve of trace $(m-2)$. Therefore, in both cases, an abelian surface of type $[-m,-m+2]$ cannot exist. 

--- If  $\{ 2 q^{1/2} \} <\sqrt 2-1 $ and $p\vert m$, or $q= 7^3$, there exists a curve of type $[-m+1,-m+1]$: the quadratic twist of the curve of type $[m-1,m-1]$ in the proof of Theorem \ref{Jq2}.
\end{proof}

\end{document}